\documentclass[11pt,reqno,oneside]{amsart}
\usepackage[top=1.2in, bottom=1.1in, left=1.2in,right=1.2in]{geometry}

\usepackage{amsmath,amsfonts,amssymb,amscd,amsthm,amsbsy,amscd}
\usepackage{bbm, epsf,calc,enumerate,color,graphicx,verbatim,cite,setspace}
\usepackage{tikz}
\usepackage{tikz-cd}
\usetikzlibrary{arrows,matrix,positioning}
\usetikzlibrary{decorations.pathreplacing,decorations.pathmorphing,calc}
\usepackage{tkz-graph}
\usepackage{thmtools}
\usepackage{thm-restate}
\usepackage[margin=1cm]{caption}
\usepackage{mathtools}
\usepackage[colorlinks=true, linkcolor=blue, citecolor=red, pdfborder={0 0 0}]{hyperref}
\usepackage{cleveref}
\usepackage{accents}
\usepackage{upgreek}

\usepackage{wrapfig}
\usepackage{multirow}
\usepackage{float}
\usepackage{blkarray}
\usepackage{enumitem}

\usepackage{setspace}

\theoremstyle{plain}
 \newtheorem{theorem}{Theorem}[section]
 \newtheorem{lemma}[theorem]{Lemma}
 \newtheorem{proposition}[theorem]{Proposition}
 \newtheorem{corollary}[theorem]{Corollary}
 \newtheorem{conjecture}[theorem]{Conjecture}
 \newtheorem{claim}[theorem]{Claim}

\theoremstyle{def}
 \newtheorem{assumption}[theorem]{Assumption}

\theoremstyle{remark}
\newtheorem{remark}[theorem]{Remark}
 \newtheorem{definition}{Definition}[section]





\numberwithin{equation}{section}

\newcommand\nc\newcommand
\newcommand\dmo\DeclareMathOperator

\dmo{\re}{Re}
\dmo{\im}{Im}
\dmo{\sign}{sign}
\dmo{\spt}{spt}
\dmo{\supp}{supp}
\dmo{\sym}{Sym}
\dmo{\R}{\mathbb{R}}
\dmo{\C}{\mathbb{C}}
\dmo{\N}{\mathbb{N}}
\dmo{\Z}{\mathbb{Z}}
\dmo{\Id}{{Id}}
\nc{\erdos}{Erd\H os }
\nc{\er}{Erd\H os-R\'enyi } 
\def \<{\langle}
\def \>{\rangle}
\def \etc {,\ldots,}
\def \lf {\lfloor}
\def \rf {\rfloor}
\nc{\norm}[1]{\left \| #1 \right \|}
\nc{\expo}[1]{\exp \left( #1 \rule{0mm}{3mm}\right)}
\DeclarePairedDelimiter\parentheses{\lparen}{\rparen}

\dmo{\e}{\mathbb{E}}
\dmo{\var}{Var}
\dmo{\pr}{\mathbb{P}}
\dmo{\un}{\mathbbm{1}}
\nc{\eqd}{\,{\buildrel d \over =}\,}
\dmo{\ber}{Ber}
\dmo{\sber}{Ber_{\pm}}
\dmo{\bin}{Bin}
\nc{\bad}{\mathcal{B}}
\nc{\event}{\mathcal{E}}
\nc{\good}{\mathcal{G}}
\nc{\pro}[1]{\mathbb{P}\parentheses*{#1 \rule{0mm}{0mm}}}
\nc{\pros}[2][]{\mathbb{P}_{#1} \parentheses*{ #2 \rule{0mm}{3mm}}}
\nc{\ero}[1]{\mathbb{E}\parentheses*{ #1 \rule{0mm}{3mm}}}
\nc{\eros}[2][]{\mathbb{E}_{#1} \parentheses*{ #2 \rule{0mm}{3mm}}}
\nc{\set}[1]{\left\{ #1 \right\}}
\nc{\cond}{\left\vert\vphantom{\frac{1}{1}}\right.}

\dmo{\tr}{tr}
\dmo{\rank}{rank}
\dmo{\corank}{corank}
\def \tran {\mathsf{T}}

\nc{\mat}[4][]{\begin{pmatrix} #1 & #2 \\ #3 & #4\end{pmatrix}}
\dmo{\idd}{\mat[1]{0}{0}{1}}
\dmo{\jdd}{\mat[0]{1}{1}{0}}
\nc{\twov}[2][]{\begin{pmatrix} #1 \\ #2 \end{pmatrix}}
\dmo{\eone}{\twov[1]{0}}
\dmo{\etwo}{\twov[0]{1}}
\nc{\Span}{\operatorname{span}}

\nc{\eps}{\varepsilon}
\nc{\ep}{\epsilon}
\nc{\tx}{\tilde{x}}
\nc{\ty}{\tilde{y}}
\nc{\tpi}{\tilde{\pi}}
\nc{\mA}{\mathcal{A}}
\nc{\mB}{\mathcal{B}}
\nc{\mC}{\mathcal{C}}
\nc{\mD}{\mathcal{D}}
\nc{\mE}{\mathcal{E}}
\nc{\mF}{\mathcal{F}}
\nc{\mG}{\mathcal{G}}
\nc{\mH}{\mathcal{H}}
\nc{\mI}{\mathcal{I}}
\nc{\mJ}{\mathcal{J}}
\nc{\mK}{\mathcal{K}}
\nc{\mL}{\mathcal{L}}
\nc{\mM}{\mathcal{M}}
\nc{\mN}{\mathcal{N}}
\nc{\mO}{\mathcal{O}}
\nc{\mP}{\mathcal{P}}
\nc{\mQ}{\mathcal{Q}}
\nc{\mR}{\mathcal{R}}
\nc{\mS}{\mathcal{S}}
\nc{\mT}{\mathcal{T}}
\nc{\mU}{\mathcal{U}}
\nc{\mV}{\mathcal{V}}
\nc{\mW}{\mathcal{W}}
\nc{\mX}{\mathcal{X}}
\nc{\mY}{\mathcal{Y}}
\nc{\mZ}{\mathcal{Z}}

\DeclareMathSymbol{\wtil}{\mathord}{largesymbols}{"65}
\newcommand\lowerwtil{%
  \text{\smash{\raisebox{-1.3ex}{%
    $\wtil$}}}}
\newcommand\witil[1]{%
  \mathchoice
    {\accentset{\displaystyle\lowerwtil}{#1}}
    {\accentset{\textstyle\lowerwtil}{#1}}
    {\accentset{\scriptstyle\lowerwtil}{#1}}
    {\accentset{\scriptscriptstyle\lowerwtil}{#1}}
}

\DeclareMathSymbol{\widehatsym}{\mathord}{largesymbols}{"62}

\DeclareMathSymbol{\widechecksym}{\mathord}{largesymbols}{"71}

\nc{\tM}{\witil{M}}
\nc{\tR}{\witil{R}}
\nc{\tS}{\witil{S}}
\nc{\tT}{\witil{T}}
\nc{\tX}{\witil{X}}
\nc{\tY}{\witil{Y}}
\nc{\tZ}{\witil{Z}}
\nc{\tA}{\witil{A}}
\nc{\tB}{\witil{B}}
\nc{\tD}{\witil{D}}

\nc{\tevent}{{\witil{\event}}}
\nc{\tbad}{\witil{\bad}}

\nc{\hx}{\hat{x}}
\nc{\hy}{\hat{y}}
\nc{\hz}{\hat{z}}

\dmo{\degr}{deg}

\dmo{\ones}{\mathbf{1}}
\nc{\eye}{{\mathbf{I}_2}}
\nc{\jay}{{\mathbf{J}_2}}
\nc{\rowpair}{(i_1,i_2)}
\dmo{\sls}{sls}
\dmo{\SLS}{SLS}
\dmo{\codeg}{codeg}
\dmo{\discrep}{discrep}
\dmo{\sparse}{sp}
\dmo{\co}{co}
\dmo{\Co}{Co}
\dmo{\Ex}{Ex}
\dmo{\ex}{ex}
\dmo{\exbar}{\overline{ex}}
\nc{\nbs}[4][]{\mN^{#1}_{(#2,#3)}(#4)}
\dmo{\Steps}{Steps}
\dmo{\steps}{steps}
\dmo{\Flats}{Flats}
\dmo{\Cross}{Cross}
\dmo{\Perm}{Perm}
\dmo{\disc}{disc}
\nc{\pe}{p}
\dmo{\simple}{simple}
\dmo{\reg}{reg}
\dmo{\ee}{edge}
\dmo{\thin}{thin}
\dmo{\Fix}{Frozen}
\dmo{\schur}{\circ}

\tikzset{
	My Style/.style={
        circle,
        draw,
        fill          = black!50,
        inner sep     = 0pt,
        minimum width =4 pt
    }   
}

\title[Random regular digraphs: singularity of the adjacency matrix]{On the singularity of adjacency matrices\\ for random regular digraphs}
\author{Nicholas A. Cook}
\address{Department of Mathematics, UCLA, Los Angeles, CA 90095-1555}
\email{nickcook@math.ucla.edu}

\let\oldtocsubsection=\tocsubsection
\renewcommand{\tocsubsection}[2]{\hspace*{1.1cm}\oldtocsubsection{#1}{#2}}

\begin{document}

\maketitle

\begin{abstract}
We prove that the (non-symmetric) adjacency matrix of a uniform random $d$-regular directed graph on $n$ vertices is asymptotically almost surely invertible, assuming 
$\min(d,n-d)\ge C\log^2n$ for a sufficiently large constant $C>0$.
The proof makes use of a coupling of random regular digraphs formed by ``shuffling" the neighborhood of a pair of vertices, as well as concentration results for the distribution of edges, proved in \cite{Cook:discrep}.
We also apply our general approach to prove a.a.s.\ invertibility of Hadamard products $\Sigma\schur \Xi$, where $\Xi$ is a matrix of iid uniform $\pm1$ signs, and $\Sigma$ is a 0/1 matrix whose associated digraph satisfies certain ``expansion" properties. 
\end{abstract}

\setcounter{tocdepth}{3}
\tableofcontents

\section{Introduction}
\label{sec_intro}

For $n\ge 1$ and $d\in [n]$, let $\mM_{n,d}$ be the set of $n\times n$ matrices with entries in $\{0,1\}$ satisfying the constraint that all row and column sums are equal to $d$. 
(For instance, we have that $\mM_{n,1}$ is the set of $n\times n$ permutation matrices.)
One may interpret the elements of $\mM_{n,d}$ as the adjacency matrices of $d$-regular digraphs -- that is, directed graphs on $n$ labeled vertices with each vertex having $d$ in-neighbors and $d$ out-neighbors (allowing self-loops). 
One can also identify $\mM_{n,d}$ with the set of $d$-regular bipartite graphs on $n+n$ vertices in the obvious way. 

We denote by $M$ a uniform random element of $\mM_{n,d}$, and refer to $M$ as an ``rrd matrix" (for ``random regular digraph"). 
Our objective in this paper is to determine whether $M$ is invertible with high probability when $n$ is large and for some range of the parameter $d$.
Before stating our main result, we give an overview of related work on other random matrix models.

\subsection{Background}
\label{sec_background}

Much work on the singularity of random matrices has focused on \emph{iid sign matrices} $\Xi$, 
whose entries are iid uniform $\pm1$ Bernoulli random variables.
It is already a non-trivial problem to prove that $\Xi$ is invertible with probability tending to 1; 
this was first accomplished by Koml\'os in the works \cite{Komlos, Komlos2} from the 1960s.
His proof was later refined to give the following quantitative bound:
\begin{theorem}[Koml\'os \cite{Komlos:net}]
\label{thm_komlos}
Let $\Xi$ be an $n\times n$ matrix of iid uniform signs. 
Then 
\begin{equation}	\label{komlosbd}
\pr\big( \det(\Xi) = 0 \big) = O\big( n^{-1/2}\big).
\end{equation}
\end{theorem}
The asymptotic notation in \eqref{komlosbd} and throughout this paper is with respect to the large $n$ limit -- see Section \ref{sec_notation} for our notational conventions. 

A key ingredient in the proof of Theorem \ref{thm_komlos} was a bound of Littlewood-Offord type due to Erd\H os (Theorem \ref{thm_erdos} below) from the seemingly unrelated field of \emph{additive combinatorics}. 
This inspired a sequence of works improving \eqref{komlosbd} to exponential bounds
\begin{equation}	\label{expobd}
\pr\big( \det(\Xi) = 0 \big) \ll c^n
\end{equation}
for some constant $c<1$
by making heavier use of additive combinatorics machinery.
Specifically, the base $c=.999$ was obtained by Kahn, Koml\'os and Szemer\'edi in \cite{KKS}, and was lowered to $c=3/4+o(1)$ by Tao and Vu \cite{TaVu:sing}, and to $c=1/\sqrt{2}+o(1)$ by Bourgain, Vu and Wood \cite{BVW}.
The latter two works relied on the \emph{inverse Littlewood-Offord theory} developed in \cite{TaVu:sing}.
These bounds still fall short of the folklore conjecture
\begin{equation}		\label{conjbd}
\pr\big( \det(\Xi) = 0 \big) =n^2(1+o(1))2^{1-n}
\end{equation}
which has been stated in \cite{Komlos2, KKS}.
The lower bound in \eqref{conjbd} is easily proved by considering the event that $\Xi$ has a pair of rows or columns that are parallel.

One source of motivation for controlling the singularity probability is its relation to the problem of proving limit laws for the distribution of eigenvalues.
Define the (rescaled) \emph{empirical spectral distribution} of $\Xi$ to be the random probability measure
$$\mu_{\frac{1}{\sqrt{n}}\Xi}:= \frac1n\sum_{i=1}^n \delta_{\lambda_i(\frac{1}{\sqrt{n}}\Xi)},$$
distributed uniformly over the eigenvalues of the normalized matrix $\frac{1}{\sqrt{n}}\Xi$.
In \cite{TaVu:esd} Tao and Vu proved the \emph{circular law} for $\Xi$, which states that almost surely, as $n\rightarrow \infty$, $\mu_{\frac{1}{\sqrt{n}}\Xi}$ converges weakly to the uniform measure on the unit disc in $\C$. 
They actually proved a \emph{universality principle}, which implies that the circular law holds for any matrix with iid entries having mean 0 and variance 1.

The main technical hurdle in the proof of the circular law was to obtain good lower bounds on the least singular value $\sigma_n(\Xi)$ holding with high probability. (Actually, it was necessary to do this for arbitrary scalar shifts $\frac{1}{\sqrt{n}}\Xi-zI$.)
Proving lower bounds on $\sigma_n(\Xi)$ is an extension of the singularity probability problem -- indeed, the latter is to bound $\pr(\sigma_n(\Xi)=0)$.
Polynomial lower bounds on the least singular value of a general class of iid matrices were first obtained by Rudelson in \cite{Rudelson}, and were subsequently improved by Tao and Vu \cite{TaVu:cond, TaVu:smooth} and Rudelson and Vershynin \cite{RuVe}.

See \cite{BoCh} for a survey of the circular law and related work.
The assumption of joint independence of the entries has been relaxed in some directions. 
For instance, the circular law is established for matrices with log-concave isotropic unconditional laws by Adamczak and Chafa\"{\i} in \cite{AdCh}.
Together with Wolff in \cite{ACW}, the same authors extend the circular law to matrices with exchangeable entries satisfying some moment bounds. 
(Note that while the rows and columns of the rrd matrix $M$ are exchangeable, the individual entries are not.)

Apart from iid matrices, a lot of activity has concentrated on random matrix models with constraints on row and column sums. 
In \cite{BCC}, Bordenave, Caputo and Chafa\"{\i} proved the circular law for random Markov matrices, obtained by normalizing the rows of an iid matrix with continuous entry distributions.  
On the discrete side, in \cite{Nguyen} Nguyen proved that a uniform random 0/1 matrix constrained to have all row-sums equal to $n/2$ (say $n$ is even) is invertible with probability $1-O_C(n^{-C})$.
Nguyen and Vu subsequently proved the circular law for a more general class of random discrete matrices with constant row sums \cite{NgVu}.

The approach in \cite{Nguyen} and \cite{NgVu} was to use a conditioning trick, which we now briefly sketch.
As in \cite{Nguyen}, assume $n$ is even, and let $Q$ be a uniform random $n\times n$ 0/1 matrix with all row sums equal to $n/2$.
Suppose we want to control the probability that some property $P$ holds for the first row $R_1$ of $Q$. 
We draw $Y_1\in \set{0,1}^n$ uniformly at random, and let $\event$ be the event that the components of $Y_1$ sum to $n/2$. 
We have $Y_1\,\big|\, \event \eqd R_1.$
Moreover, one can easily show that $\pr(\event) \gg n^{-1/2}$.
It follows that we can bound
\begin{align}
\pr(\,\mbox{$P$ holds for $R_1$}\,) &= \pro{ \mbox{$P$ holds for $Y_1$} \,\big|\, \event }\notag \\
&\le \frac{\pro{ \mbox{$P$ holds for $Y_1$}}}{\pr(\event)} \notag\\
&\ll n^{1/2}\pro{ \mbox{$P$ holds for $Y_1$}}.	\label{comparebd}
\end{align}
This last term can be controlled using the existing theory for iid matrices (the loss of a factor $O(n^{1/2})$ turns out to be acceptable).

The results from \cite{BCC}, \cite{Nguyen} and \cite{NgVu} still relied on the independence between rows. 
For the rrd matrix $M$ considered in the present work there is no independence among rows or columns.
In particular, an approach by conditioning iid variables as in \cite{Nguyen} can not treat each row separately,
and instead must condition on the event that the entire iid matrix is in $\mM_{n,d}$. 
The probability of this event can be estimated using asymptotic enumeration results. 
Letting $p:=d/n$, we draw a random 0/1 matrix $M_p$ with iid Bernoulli($p$) entries, and let
\begin{equation}		\label{eventd}
\event_{n,d}=\set{M_p\in \mM_{n,d}}.
\end{equation}
Then $M_p|\event_{n,d} \eqd M$.
We have 
\begin{equation}	\label{psimp}
\pr(\event_{n,d}) \sim \sqrt{2\pi d(n-d)}\expo{-n\log\left(\frac{2\pi d(n-d)}{n}\right)}
\end{equation}
which follows from an asymptotic formula for the cardinality of $\mM_{n,d}$, established for the sparse case $d=np=o(\sqrt{n})$ by McKay and Wang in \cite{McWa} and for the dense range $\min(d,n-d) \gg n/\log n$ by Canfield and McKay in \cite{CaMc}.

Although enumeration results for the range $\sqrt{n}\ll d\ll n/\log n$ are unavailable as of this writing (though it is natural to conjecture that the formula \eqref{psimp} extends to hold in this range), in \cite{Tran} Tran used an argument from \cite{ShUp} of Shamir and Upfal to show that for $d=\Omega(\log n)$, 
\begin{equation}		\label{psimplb}
\pr(\event_{n,d}) \ge \expo{-O\big(n\sqrt{d}\big)}.
\end{equation}
While weaker than \eqref{psimp}, this lower bound was enough to prove the quarter-circular law for the singular value distribution of $M$ using the conditioning trick (in fact Tran treated the more general case of rectangular 0/1 matrices with constant row and column sums, for which he proved the Marchenko--Pastur law).
The semi-circular law was established for undirected random regular graphs with $d\rightarrow \infty$ by a similar approach in \cite{TVW}.
It is worth noting that the Marchenko--Pastur and semi-circular laws were also obtained in \cite{DuPa} and \cite{DuJo} for 
the sparse regime $\omega(1)\le d\le n^{o(1)}$, 
using the fact that $d$-regular graphs converge locally (in a quantitative Benjamini--Schramm sense) to $d$-regular trees.

Hence, we see that with \eqref{psimplb} we are limited to importing properties of the iid matrix $M_p$ that hold with probability $1-O(\exp(-Cn\sqrt{np}))$ for some sufficiently large $C$,
and this can be slightly relaxed by using the formula \eqref{psimp} for the appropriate range of $d$. 
We note in particular that the results of the present work cannot be obtained by the conditioning trick.

On the continuous side, a similar conditioning approach was used to study uniform random doubly stochastic matrices in \cite{CDS} and \cite{Nguyen:uds}.
In \cite{CDS}, Chatterjee, Diaconis and Sly noted that this distribution can be obtained by conditioning a matrix with iid exponentially distributed entries. They proved the quarter circular law by similar lines to \cite{Tran}, relying on another asymptotic formula of Canfield and McKay from \cite{CaMc2} for the volume of the Birkhoff polytope.
Nguyen built on this work in \cite{Nguyen:uds} to prove the circular law for this model.

\subsection{Main results and conjectures}
\label{sec_main}

Our main result is an analogue of Koml\'os' Theorem \ref{thm_komlos} for rrd matrices, assuming that the matrix is not too sparse or too dense. 
Specifically, we assume that 
$\min(d,n-d) \ge C_0\log^2n$ for a sufficiently large constant $C_0>0$.
Our approach is by couplings rather than by the conditioning trick described above.
We give more detail and motivation for the proof strategy in Section \ref{sec_strategy} below.

\begin{theorem}[Main result]
\label{thm_main}
There are absolute constants $C_0,c_0>0$ such that the following holds. 
Assume $\min(d,n-d) \ge C_0\log^2n$, and let $M$ be a uniform random element of $\mM_{n,d}$. 
Then
\begin{equation}	\label{main_bound}
\pr\big(\det(M) = 0\big) = O(d^{-c_0}).
\end{equation}
\end{theorem}

\begin{remark}
The proof shows that we may take 
$c_0=1/18$, though we do not expect this bound to be optimal (see Conjectures \ref{conj_expo} and \ref{conj_poly} below).
\end{remark}

\begin{remark}[Lower bound on $d$]	\label{rmk_dlb0}
It is possible that our argument could be extended to only assume $\min(d,n-d)\ge C_0 \log n$, but a new approach will be certainly necessary beyond that -- see Remark \ref{rmk_dlb}.
\end{remark}

\begin{remark}	
\label{rmk_assumed}
One can easily show that a matrix $M\in \mM_{n,d}$ is invertible if and only if the ``complementary" matrix $M'$ with entries $M'(i,j) = 1-M(i,j)$
is invertible. 
Hence, in the proof of Theorem \ref{thm_main} we may assume that $p:= \frac{d}{n}\le \frac12$.
\end{remark}

\begin{remark}	\label{rmk:LLKT-JY}
Very recently (after the final version of this manuscript was sent for publication) an extension of the bound \eqref{main_bound} to lower values of $d$ has been accomplished in \cite{LLKT-JY}, along with an improvement in the exponent $c_0$.
Specifically, they are able to show that for some absolute constants $C,c>0$, 
$\pro{\det(M) = 0} =O((\log^3d )/\sqrt{d})$ if $C\le d\le cn/\log^2n$. 
Together with Theorem \ref{thm_main} this implies that a uniform random element $M\in \mM_{n,d}$ is invertible with probability tending to 1 as $n\rightarrow \infty$ if $\min(d,n-d)$ grows to $\infty$ at any speed, rather than at speed at least $\log^2n$.
See Remark \ref{rmk:LLKT-JY2} for some additional comments on this result.
\end{remark}

We believe that when $d$ is of linear size, the singularity probability is exponentially small, similarly to the bound \eqref{expobd} for iid sign matrices.

\begin{conjecture}	\label{conj_expo}
Fix $\pe_0\in \big(0,\frac12\big)$ and assume $\min(d,n-d)\ge \pe_0n$. 
Then 
$$\pr\big(\det(M) = 0\big) \le Ce^{-cn}$$
for constants $C,c>0$ depending only on $\pe_0$. 
\end{conjecture}

We also conjecture that rrd matrices are invertible with high probability for much smaller values of $d$:
\begin{conjecture}	\label{conj_poly}
There are absolute constants $C,c>0$ such that for any $3\le d\le n-3 $ we have 
$$\pr\big(\det(M) = 0\big)\le Cn^{-c}.$$
\end{conjecture}
This mirrors a similar conjecture of Vu in \cite{Vu:discrete} on the adjacency matrices of undirected $d$-regular graphs, which are the symmetric analogue of $M$.
When $d$ is bounded, considering the event that two columns of $M$ are parallel shows that we cannot hope for better than a polynomial bound on the singularity probability. 
$M$ is obviously invertible when $d=1$ as it is a permutation matrix in this case. 
On the other hand, it is not hard to show that for $d=2$, $M$ is \emph{singular} asymptotically almost surely.

Next we give a consequence of Theorem \ref{thm_main} for random sign matrices. 
Note that if we draw an iid matrix of signs $\Xi$ as in Theorem \ref{thm_komlos} and condition on the event that all rows and columns sum to 0, the resulting matrix $\Xi_0$ will be singular with null vector $\ones=(1,1,\dots,1)\in \R^n$.
It is an easy consequence of Theorem \ref{thm_main} that this is usually the only obstruction for invertibility.

\begin{corollary}
Assume $n$ is even, and let $\Xi_0$ be an $n\times n$ matrix of iid uniform $\pm1$ signs conditioned to have each row and column sum to 0.
Then with high probability, $\ker(\Xi_0)=\langle \ones\rangle$.
\end{corollary}

\begin{proof}
Let 
$M=\frac12\big(\Xi_0+\ones\ones^\tran\big).$
Then $M$ is an rrd matrix with $d=n/2$.
For $x\in \R^n$, write
$$x=\overline{x}\ones + x_0$$
with $\overline{x}=\frac1n\sum_ix(i)$ and $x_0$ the orthogonal projection of $x$ to $\langle \ones\rangle^\perp$, the space of mean-zero vectors. 
One then verifies that $x\in \ker(\Xi_0)$ if and only if $x_0\in \ker(M)$, and the result follows from Theorem \ref{thm_main}.
\end{proof}

Our next result concerns \emph{signed rrd matrices}.
Let $\mM^\pm_{n,d}$ denote the set of $n\times n$ matrices $M_\pm$ with entries in $\set{+1,0,-1}$ satisfying the constraints 
\begin{equation}	\label{pmdef}
d= \sum_{i=1}^n \big|M_\pm(i,k)\big| = \sum_{j=1}^n \big|M_\pm (k, j)\big|
\end{equation}
for all $k\in [n]$. 
We have the following analogue of Theorem \ref{thm_main}:

\begin{theorem}[Signed rrd matrices are invertible a.a.s.]
\label{thm_pm}
Assume 
$C\log^2n\le d\le n$ for a sufficiently large constant $C>0$, and let $M_{\pm}$ be a uniform random element of $\mM^{\pm}_{n,d}$. 
Then
\begin{equation}	\label{pmbound}
\pr\big( \det(M_\pm) = 0\big) = O(d^{-1/4}).
\end{equation}
\end{theorem}

Note that in contrast to Theorem \ref{thm_main}, in the above result we don't need to assume an upper bound on $d$ (apart from the trivial one). 
This is because of the additional randomness of the Bernoulli signs in $\Xi$: as $d$ approaches $n$, $M$ approaches the non-random (and singular) matrix of all 1s, while $M_{\pm}$ approaches an iid sign matrix.
In particular, by considering $d=n$ we see that Theorem \ref{thm_pm} is a generalization of Koml\' os' Theorem \ref{thm_komlos}, up to a small loss in the exponent of $n$ from the bound \eqref{komlosbd}.

The signed rrd matrix $M_\pm$ is easier to work with than the unsigned rrd matrix $M$ due to the following alternative description.
Letting $M$ be the rrd matrix with entries
\begin{equation}	\label{unsigned}
M(i,j) := |M_\pm(i,j)|
\end{equation}
we have $M_\pm \eqd M\schur \Xi$, where $\Xi$ is an iid sign matrix independent of $M$. 
(Here $\schur$ denotes the Hadamard (or Schur) product, so that $M_\pm(i,j)=M(i,j)\Xi(i,j)$ for each $i,j\in [n]$.)
We refer to $M$ as the ``base" or ``support" of the signed rrd matrix $M_\pm$.
Roughly speaking, our approach to proving Theorem \ref{thm_pm} will be to condition on a ``good" realization of the base rrd matrix $M$ and proceed using only the randomness of the iid signs. 
We will then have to show that such good configurations $M$ occur with high probability. 

The conditions of a good configuration are most naturally stated in terms of the $d$-regular digraph $\Gamma=(V,E)$ which has $M$ as its adjacency matrix.
We identify $V$ with $[n]$, and $E\subset [n]^2$ is such that for all $i,j\in [n]$,
$M(i,j) = 1_E(i,j).$
We associate row and column indices with vertices of $\Gamma$.
For $i\in [n]$, let
\begin{equation}		\label{defnb1}
\mN_M(i) := \big\{j\in [n]: M(i,j)=1\big\}
\end{equation}
so that $\mN_M(i)$ and $\mN_{M^\tran}(i)$ are the out- and in-neighborhoods of the vertex $i$, respectively, in $\Gamma$. 
For $S\subset [n]$ we denote
\begin{equation}	\label{defnb2}
\mN_M(S) := \bigcup_{i\in S}\mN_M(i).
\end{equation}
For $A,B\subset [n]$ we let
\begin{align}
e_M(A,B) &:= \big| (A\times B) \cap E\big| \notag\\
&= \sum_{i\in A} \sum_{j\in B} M(i,j)		\label{defedge}
\end{align}
count the number of directed edges passing from $A$ to $B$. 

Roughly speaking, the base matrix $M$ is a good configuration if the associated digraph satisfies certain \emph{expansion} properties. 
In Section \ref{sec_discrepancy} we prove that the rrd matrix $M$ satisfies all of the necessary properties with overwhelming probability if 
$d=\omega(\log n)$.
The proofs rely on sharp tail bounds for the edge counts $e_M(A,B)$, which were proved in \cite{Cook:discrep}.
The proof of Theorem \ref{thm_main} for the unsigned rrd matrix $M$ will rely more heavily on the expansion properties of Section \ref{sec_discrepancy}.

It turns out that to prove Theorem \ref{thm_pm} by this approach we do not need all of the expansion properties enjoyed by $M$.
Hence, we will actually prove the following result, where $M$ is replaced by a general 0/1 matrix $\Sigma$, and the event $\good(d)$ distills the required expansion properties.
In particular, Theorem \ref{thm_gen} is independent of the results in \cite{Cook:discrep}.

\begin{theorem}[0/$\pm$1 matrices with expanding support are invertible a.a.s.]
\label{thm_gen}
Let $\Sigma$ be a random or deterministic 0/1 matrix.
For $d\in [n]$, let $\good(d)$ be the event that $\Sigma$ enjoys the following \emph{expansion properties} with constants $c_1,c_2,C_2>0$ and a parameter $\kappa_3\ge 1$:
\begin{enumerate}[start=0]
\item (Minimum degree) Every row and column of $\Sigma$ has at least $d$ nonzero entries. That is, for all $i\in [n]$,
\begin{equation}	\label{deglb}
|\mN_\Sigma(i)|,\, |\mN_{\Sigma^\tran}(i)| \ge d.
\end{equation}
\item (Expansion of small sets) There is some constant 
$c_1>0$ such that for all $\gamma\in (0,c_1]$, 
for all $S\subset [n]$ such that $|S|\le \frac{\log n}{2\gamma }\frac{n}{d}$, we have
$$|\mN_\Sigma(S)|,\, |\mN_{\Sigma^\tran}(S)| \ge \frac{\gamma}{\log n}d|S|.$$
\item (No large sparse minors) There are constants 
$C_2,c_2>0$ 
such that for all $A,B\subset[n]$ satisfying
$|A|,|B|\ge C_2\frac{n}{d}\log n$
we have
$$e_\Sigma(A,B) \ge c_2\frac{d}{n}|A||B|.$$
\item (No thin dense minors) There is a parameter 
$\kappa_3\in [1,\infty)$, 
possibly depending on $n$, such that for any $S,B\subset[n]$,
$$e_\Sigma(S,B),\, e_{\Sigma}(B,S) \le \kappa_3d|S|.$$
(In particular, taking $S$ to be a singleton we have the degree bounds 
$$|\mN_\Sigma(i)|,|\mN_{\Sigma^\tran}(i)| \le \kappa_3d$$
to complement \eqref{deglb} above.)
\end{enumerate}
Let $\Xi $ be an iid sign matrix independent of $\Sigma$, and put $H=\Sigma\schur\Xi$.
There is a constant $C_0'>0$ depending on $c_1,c_2,C_2$ such that if $d\ge C_0'\log^2n$,
then 
\begin{equation}	\label{genbound}
\pr\big(\set{\det(H)=0} \wedge \good(d)\big) = O(\kappa_3d^{-1/4}).
\end{equation}
In particular, if properties (0)-(3) hold a.a.s.\ for $\Sigma$ with 
$d\ge C_0'\log^2n$
and $\kappa_3=o(d^{1/4})$, then $H$ is invertible 
a.a.s.
\end{theorem}

In Section \ref{sec_discrepancy} we will show that with $\Sigma=M$ the rrd matrix, the event $\good(d)$ holds with overwhelming probability 
for some constants $c_1,c_2,C_2>0$ (we can take $\kappa_3=1$ in this case) assuming $d=\omega(\log n)$, at which point Theorem \ref{thm_pm} follows from Theorem \ref{thm_gen}.

The proof of Theorem \ref{thm_gen} will follow the general outline of Theorem \ref{thm_main}, but each stage will be easier due to the independence of the entries of $\Xi$. 
Hence, we believe it will benefit the reader to first see the arguments for $H$ as warmups to the more complicated arguments involving couplings for the rrd matrix $M$, and have structured the paper accordingly.
However, nothing from the proof of Theorem \ref{thm_gen} is needed for the proof of Theorem \ref{thm_main}, so the reader who is only interested in the proof of the main theorem can skip the sections devoted to $H$ (namely, Sections \ref{sec_pm} and \ref{sec_structpm}).

\subsection{The general strategy}
\label{sec_strategy}

Now we give a high level discussion 
of our couplings approach to proving invertibility of an rrd matrix $M$.
The strategy is similar in spirit to the one used by Rudelson and Vershynin in the recent work \cite{RuVe:pert} on the least singular value of perturbations of deterministic matrices by Haar unitary or orthogonal matrices. 
As the rrd matrix $M$ has discrete distribution, the couplings we define will be of a very different nature from the ones considered in that work.
Nevertheless, on a conceptual level at least, we cannot overvalue the influence \cite{RuVe:pert} has had on our approach for dealing with dependent random variables.

In order to improve on the strategy of conditioning on an iid Bernoulli($p$) matrix $M_p$ as in \eqref{eventd}, we would like to show that the events 
$\set{\det(M_p)=0}$ and $\set{M_p\in \mM_{n,d}}$
are approximately independent in some sense. 
Indeed, proceeding as in \eqref{comparebd} gives
\begin{align}
\pr(\,\det(M)=0\,) &= \pro{\, \det(M_p)=0 \,\big|\, M_p\in \mM_{n,d} \,}\notag \\
&\le \frac{\pro{ \det(M_p)=0}}{\pro{M_p\in \mM_{n,d}}} \label{worstcase}
\end{align}
which is only sharp for the worst case that we have the containment
\begin{equation}	\label{farfrom}
\set{\det(M_p)=0} \subset \set{M_p\in \mM_{n,d}}.
\end{equation}
Of course, \eqref{farfrom} is likely far from the truth.
This motivates us to better understand the structure of the set $\mM_{n,d}$; specifically, we try to identify \emph{symmetries} of this set.
If we can identify a large class of operations $\Phi: \mM_{n,d}\rightarrow \mM_{n,d}$ which leave the distribution of a uniform random element $M\in \mM_{n,d}$ invariant, then we could select such an operation $\Phi $ \emph{at random} from this class and form a new rrd matrix
$\tM= \Phi (M).$
Now to bound the event that some property $P$ holds for $M$, we may replace $M$ with $\tM$:
\begin{align*}
\pro{\,\mbox{$P$ holds for $M$}\,} &= \pro{\,\mbox{$P$ holds for $\tM$}\,}\\
&=\e \pro{\,\mbox{$P$ holds for $\tM$}\,\big|\, M}
\end{align*}
and proceed to bound the inner probability using only the randomness we have ``injected" via the map $\Phi $.
This approach can be very powerful if we can design the map $\Phi $ to involve a large number of independent random variables.

This strategy was used in \cite{RuVe:pert} to obtain bounds of the form
\begin{equation}	\label{rvbound}
\pro{ \sigma_n(D+U) \le t} \ll t^cn^C
\end{equation}
for some absolute constants $C,c>0$, where $D$ is a deterministic matrix (satisfying some additional hypotheses) and $U$ is a Haar-distributed unitary or orthogonal matrix. 
Since the random matrices in this case are drawn from a group, there is no shortage of symmetries to consider for injecting independence.
Furthermore, the availability of \emph{continuous} symmetries allowed for the injection of random variables possessing smooth bounded density (such as iid Gaussians).
This gave quick access to anti-concentration or ``small ball" estimates, which play a fundamental role in all currently known (to this author at least) proofs of invertibility for random matrices.

The bound \eqref{rvbound} had implications for the Single Ring Theorem, proved by Guionnet, Krishnapur and Zeitouni in \cite{GKZ}, for the limiting spectral distribution of certain random matrices with prescribed singular values; specifically, it was shown that a hypothesis in \cite{GKZ} could be disposed of.
It was also used in the proof by Basak and Dembo in \cite{BaDe} of the limiting spectral distribution for the sum of a fixed number of independent Haar unitary or orthogonal matrices. 
It was conjectured in \cite{BoCh} that the same law should hold for the random matrix
$$M_{\Perm}:=P_1+\cdots+P_d$$
where the summands are iid uniform $n\times n$ permutation matrices --
this can be viewed as a sparse version of the rrd matrix $M$.
It is possible that the least singular value of $M_{\Perm}$ could be controlled by an extension of the ideas used in the present work, though the extreme sparsity of this matrix will likely call for new ideas.

The present setting of rrd matrices is a little more complicated than the case of Haar unitaries as the distribution is discrete, and $\mM_{n,d}$ is not a group (except when $d=1$ of course, but then the problem is trivial).
The basic building block for our coupled pairs $(M,\tM)$ will be the well-known ``simple switching" operation: 
letting
\begin{equation}\label{IJ0}
\eye:=
\begin{pmatrix}
1& 0 \\
0 & 1  \end{pmatrix}, 
\hspace{.7cm}
\jay:=
\begin{pmatrix}
0& 1 \\
1 & 0  \end{pmatrix}
\end{equation}
we can replace a $2\times2$ minor of $M$ by $\eye$ if it is $\jay$ and $\jay$ if it is $\eye$ -- indeed, note that this preserves the row and column sums. If $i_1,i_2$ and $j_1,j_2$ are the row and column indices, respectively, of such a minor, then
in the associated digraph $\Gamma=(V,E)$ we are alternating between the following edge configurations at vertices $i_1,i_2,j_1,j_2$:
\begin{center}
  \includegraphics[width=80mm]{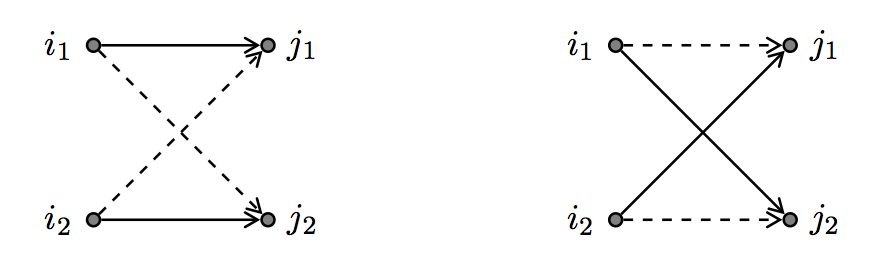}
\end{center}
where we use solid arrows to depict directed edges, and dashed arrows to indicate the absence of a directed edge (i.e.\ ``non-edges"). This forms the basis for (a simple instance of) what is known as the \emph{switching method}, which has been a successful tool in the study of random regular graphs since its introduction by McKay in \cite{McKay2}; see also section 2.4 of the survey \cite{Wormald}. 

We will want to form $\tM$ by applying several switchings at non-overlapping $2\times 2$ minors. 
Each minor is replaced with $\eye$ or $\jay$ uniformly at random, independently of all other switchings. 
We can encode the outcomes of the random switchings with iid uniform signs -- this will give us access to anti-concentration estimates for random walks (specifically Erd\H os' Theorem \ref{thm_erdos} below).
The formal construction, called the ``shuffling coupling", is given in Section \ref{sec_shuffling}.

\subsection{Organization of the paper}
\label{sec_organization}

The rest of the paper is organized as follows. 
In Section \ref{sec_ideas} we describe the ideas of the proof in more detail, reviewing the approach introduced by Koml\'os to classify potential null vectors as structured and unstructured, and illustrating our use of couplings by solving a toy problem.
Section \ref{sec_prelim} gives the formal statements and proofs for the tools that were motivated in Section \ref{sec_ideas}, namely the ``shuffling" coupling of rrd matrices, discrepancy properties for random regular digraphs (including results from \cite{Cook:discrep}), and a concentration inequality for the symmetric group due to Chatterjee.
After establishing the discrepancy properties for $M$ in Section \ref{sec_discrepancy}, we will deduce Theorem \ref{thm_pm} from Theorem \ref{thm_gen}.
In Sections \ref{sec_unstruct} and \ref{sec_struct} we bound the events that our random matrices have ``unstructured" and ``structured" null vectors, respectively -- 
for the signed matrix $H$ ``structured" will mean ``sparse", while for the rrd matrix $M$ it will mean that the null vector has a large level set.
(Note that in the recent literature on invertibility of iid matrices the term ``structured" is used for vectors whose components lie in a set that is well-approximated in some sense by a \emph{generalized arithmetic progression} -- see \cite{TaVu:sing, TaVu:cond, RuVe}.)
In each section, we first treat the signed matrix $H$ as a warmup to the more complicated arguments for $M$.
However, the reader who is only interested in the proof of the main theorem concerning the rrd matrix $M$ can skip Sections \ref{sec_pm} and \ref{sec_structpm}.

\subsection{Notation}
\label{sec_notation}

We make use of the following asymptotic notation with respect to the limit $n\rightarrow \infty$ (though the proof can easily be made effective).
$f\ll g$, $g\gg f$, $f=O(g)$, and $g=\Omega(f)$ are synonymous to the statement that $|f|\le Cg$ for all $n\ge C$ for some absolute constant $C$. $f\asymp g$ and $f=\Theta(g)$ mean $f\ll g$ and $f\gg g$.
$f=o(g)$ and $g=\omega(f)$ mean that $f/g\rightarrow 0$ as $n$ tends to infinity. 
Dependence of the implied constant on a parameter $\alpha$ is indicated with a subscript: $f=O_\alpha(g), f\ll_\alpha g$, 
etc.
$C,c, c'$, etc.\ are used to denote various unspecified absolute constants, and their values may change from line to line.
Some distinguished constants from the statements of Theorems and Propositions (such as $C_0, c_0$ in Theorem \ref{thm_main}) have numbered subscripts so that they can be more easily tracked through the arguments. 
We still allow hidden constants in asymptotic notation to depend on these numbered constants. 

Most events will be denoted by the letters $\event, \bad,$ and $\good$, where the latter two denote ``bad" and ``good" events, respectively. 
Their meaning may vary from proof to proof, but will remain fixed for the duration of each proof. 
$\un_\event$ denotes the indicator random variable for the event $\event$, and for a statement $P$, 
$\un(P):= \un_{\set{P \text{ holds}}}$.
$\e_X$ and $\pr_X$ denote expectation and probability, respectively, conditional on all random variables but $X$.

We make use of the following terminology for sequences of events. 
\begin{definition}[Frequent events]
\label{def_frequent}
An event $\event$ depending on $n$ holds
\begin{itemize}
\item \emph{asymptotically almost surely} (a.a.s.) if $\pr(\event) = 1-o(1)$, 
\item \emph{with high probability} (w.h.p.) if $\pr(\event) = 1-O(n^{-c})$ for some absolute constant $c>0$,
\item \emph{with overwhelming probability} (w.o.p.) if $\pr(\event) = 1- O_C(n^{-C})$ for any constant $C>0$.
\end{itemize}
\end{definition}

Given ordered tuples of row and column indices $(i_1,\dots, i_a)$ and $(j_1,\dots, j_b)$, we denote by 
$M_{(i_1,\dots, i_a)\times(j_1,\dots, j_b)}$
the $a\times b$ matrix with $(k,l)$ entry equal to the $(i_k,j_l)$ entry of $M$. 
(Note for instance that the sequence $(i_1,\dots,i_a)$ need not be increasing.) 
For $A,B\subset [n]$, with $M_{A\times B}$ the increasing ordering of the elements of $A,B$ is implied.
We also recall from \eqref{IJ0} our notation
\begin{equation}	\label{IJ}
\eye:=
\begin{pmatrix}
1& 0 \\
0 & 1  \end{pmatrix}, 
\hspace{.7cm}
\jay:=
\begin{pmatrix}
0& 1 \\
1 & 0  \end{pmatrix}.
\end{equation}
We use $\circ$ to denote the matrix Hadamard product; that is, for $n\times m$ matrices $M_1, M_2$, $M_1\circ M_2$ is the $n\times m$ matrix with entries $M_1(i,j)M_2(i,j)$. 
We will use this notation for row vectors as well (which is the case $n=1$).

We use notation that views a vector $x\in \R^n$ as a function
$x:[n]\rightarrow\R$. 
In particular, the $i$th component of $x$ is denoted $x(i)$. 
We also define the support of $x$ as
$$\spt(x):=\big\{i\in [n]:x(i)\ne 0\big\} = [n]\setminus x^{-1}(0).$$
We write $\R^T\subset \R^n$ for the subspace of vectors supported on $T\subset[n]$. 
We let $\ones$ denote the all-ones vector $(1\etc 1)\in \R^n$. 
``Null vector" will mean ``right null vector" unless otherwise stated.
The span of a single vector $x$ is denoted by $\langle x\rangle$.

As noted in Section \ref{sec_main}, it will be convenient to use some terminology reflecting the association of $M$ with a $d$-regular digraph $\Gamma=(V,E)$. 
In addition to the notation \eqref{defnb1}--\eqref{defedge}, 
for distinct row indices $i_1,i_2\in [n]$ we define the sets of column indices
\begin{align}
\Co_M(i_1,i_2) &= \mN_M(i_1)\cap \mN_M(i_2) = \set{ j\in [n]: M_{(i_1,i_2)\times j} = {1\choose 1}}	\label{codef}\\
\Ex_M(i_1,i_2) &= \mN_M(i_1)\setminus \mN_M(i_2) = \set{ j\in [n]: M_{(i_1,i_2)\times j} = {1\choose 0}}	\label{exdef}
\end{align}
so that
\begin{equation}
\Ex_M(i_2,i_1) = \set{ j\in [n]: M_{(i_1,i_2)\times j} = {0\choose 1}}.
\end{equation}
These three sets partition the vertex-pair neighborhood $\mN_M(\set{i_1,i_2})$. 
We denote their cardinalities with lower case: 
$\co_M(i_1,i_2)= |\Co_M(i_1,i_2)|, \; \ex_M(i_1,i_2)= |\Ex_M(i_1,i_2)|$,
the first of these being the usual (out-)codegree of vertices $i_1,i_2$.

\section{Ideas of the proof}
\label{sec_ideas}

Our general approach to Theorem \ref{thm_main} is inspired by Koml\'os' proof of the analogous theorem for iid sign matrices. 
After briefly reviewing Koml\'os' argument below, we will discuss the new ideas that are necessary to treat rrd matrices. 

\subsection{The strategy of Koml\'os}
\label{sec_komlos}

A key ingredient of Koml\'os' proof is the following ``discrete small ball estimate" for random walks due to Erd\H os. 

\begin{theorem}[Anti-concentration for random walks \cite{Erdos}]
\label{thm_erdos}
Let $x\in \R^n$ be a fixed nonzero vector, and let $\xi:[n]\rightarrow\set{\pm1}$ be a sequence of iid uniform signs. 
Then
\begin{equation}
\sup_{a\in \R}\, \pr\bigg\{ \sum_{i=1}^nx(i)\xi(i) = a \bigg\} \ll |\spt(x)|^{-1/2}
\end{equation}
where we recall the notation $\spt(x)= \big\{i\in [n]:x(i)\ne 0\big\}$.
\end{theorem}

\begin{proof}{\emph{of Theorem \ref{thm_komlos}}}.
We want to bound the bad event
\begin{equation}
\bad := \big\{ \det(\Xi )= 0 \big\} = \big\{ \exists \mbox{ nonzero } x\in \R^n: \, \Xi\, x=0\big\}.
\end{equation}
The idea is to separately consider the possibility of ``structured" and ``unstructured" null vectors $x$. 
Here the right notion of structure is \emph{sparsity}.
Say that $x\in \R^n$ is $k$-sparse if $|\spt(x)|\le k$. 

\begin{proposition}[No structured null vectors for $\Xi $]
\label{prop_structber}
For any fixed $\eta\in (0,1)$, with overwhelming probability $\Xi $ has no nontrivial $(1-\eta)n$-sparse null vectors. 
\end{proposition}

We defer the proof of this proposition to the end. 
Fix $\eta\in (0,1)$. 
We say that $x\in \R^n$ is ``structured" if $x$ is $(1-\eta)n$-sparse, and ``unstructured" otherwise.
Since $\Xi $ is identically distributed to its transpose, we may now restrict to the event $\good$ on which $\Xi $ has no structured left or right null vectors. 

For each $i\in [n]$, let $R_i$ denote the $i$th row of $\Xi $, and denote
$V_i= \Span(R_{i'}: i'\ne i).$
Define the events
$$\bad_i := \good \wedge \set{ R_i\in V_i}.$$
On $\bad\wedge \good$, $\Xi $ must have an unstructured left null vector, which implies that $\bad_i$ holds for at least $(1-\eta)n$ values of $i\in [n]$. 
By double counting we then have that
\begin{equation}
\sum_{i=1}^n \pr(\bad_i) \ge (1-\eta) n\pr(\bad\wedge \good).
\end{equation}
Since the rows of $\Xi $ are exchangeable, all of the summands on the left hand side are equal to $\pr(\bad_1)$, say, and so
\begin{equation}
\pr(\bad\wedge \good) \le \frac{1}{1-\eta} \pr(\bad_1).
\end{equation}
By our bound on $\pr(\good^c)$ from Proposition \ref{prop_structber}, it only remains to show that $\pr(\bad_1)\ll n^{-1/2}$.

We condition on the rows $R_2,\dots, R_n$ of $\Xi $, which fixes their span $V_1$.
Condition also on a unit normal vector $u\in V_1^\perp$, drawn independently of $R_1$. 
We have
\begin{equation}
\set{R_1\in V_1} \subset \set{R_1\cdot u = 0}.
\end{equation}
On $\bad_1$ we have that $u$ is perpendicular to every row of $\Xi $, and is hence a left null vector. 
By our restriction to $\good$ we may hence assume that $u$ is unstructured.
By Theorem \ref{thm_erdos} we may now use the randomness of $R_1$ to conclude the desired bound
$$\pro{ \bad_1} \le \pro{ \good\wedge \set{R_1\cdot u = 0}} \ll_\eta n^{-1/2}.$$

We turn to the proof of Proposition \ref{prop_structber}. 
Define the events
$$\event_k = \set{\exists \, x\in \R^n: \, 0<|\spt(x)|\le k, \, \Xi\, x=0}.$$
Our aim is to bound 
\begin{equation}
\pro{\event_{(1-\eta)n}} = \sum_{k=2}^{\lf (1-\eta)n\rf} \pro{ \event_k\setminus \event_{k-1}}
\end{equation}
(noting that $\event_1$ is empty).
It suffices to show that $\pr(\event_k\setminus \event_{k-1})$ is exponentially small for arbitrary fixed $2\le k\le (1-\eta)n$.

Fix $k$ in this range. 
On $\event_k\setminus \event_{k-1}$ there is a right null vector $x$ with exactly $k$ nonzero components. 
We may spend a factor ${n\choose k}$ to assume that $x$ is supported on $[k]$ (using column exchangeability). 
Now on the complement of $\event_{k-1}$, the first $k$ columns of $\Xi $ must span a space of dimension $k-1$.
It follows that there are $k-1$ linearly independent rows of the left $n\times k$ minor of $\Xi $. 
By row exchangeability we may spend another factor ${n\choose k-1}$ to assume the first $k-1$ rows are linearly independent. 
To summarize,
\begin{equation}	\label{line2}
\pro{\event_k\setminus \event_{k-1}}\le {n\choose k} {n\choose k-1} \pro{\event_k'}
\end{equation}
where 
$$\event_k':= \Big\{ \exists x\in \R^n: \, \Xi\, x=0, \, \spt(x)=[k], \, R_1,\dots, R_{k-1} \mbox{ are linearly independent}\Big\}.$$
Now note that by linear independence, on $\event_k'$ we have that $x$ is determined by the first $k-1$ rows. Conditioning on these rows fixes $x$. 
Then by the independence of the rows of $\Xi $ we have
\begin{align}
\pro{ \event_k' \,\big|\, R_1,\dots, R_{k-1}} &\le \pro{ R_i\cdot x=0 \; \forall i\in [k,n]} \notag\\
& = \pro{ R_n\cdot x = 0}^{n-k+1}.	\label{line1}
\end{align}
Since $|\spt(x)|=k$, by Theorem \ref{thm_erdos} we can bound
\begin{equation}
\pro{ R_n\cdot x = 0 } \le \min\big[ 1/2, O(k^{-1/2})\big].
\end{equation}
Combining this bound with \eqref{line1}, \eqref{line2} and the inequality 
${n\choose n- k}\le (en/(n-k))^{n-k}$ 
we conclude
\begin{equation}
\pro{\event_k\setminus \event_{k-1}} \ll \exp\bigg\{ (n-k) \left[ C+ 2\log \frac{n}{n-k} -\log \sqrt{k} \right] \bigg\}
\end{equation}
which is more than sufficiently small if $n-k\ge \eta n$ for any fixed $\eta\in (0,1)$ (in fact we can allow $n-k$ as small as 
$C' n^{-1/4}$ for a sufficiently large absolute constant $C'>0$).
\end{proof}

\subsection{Structured and unstructured null vectors}
\label{sec_divideandconquer}

It turns out that Proposition \ref{prop_structber} is robust under some zeroing out of the entries of $M$. 
Specifically, we can show an analogous result for the matrix $H=\Sigma\schur\Xi$ from Theorem \ref{thm_gen}.

\begin{proposition}[No structured null vectors for $H$]
\label{prop_structpm}
For $\eta\in (0,1]$, let $\good_\pm^{\sparse}(\eta)$ be the event that $H$ has no nontrivial $(1-\eta)n$-sparse left or right null vectors. 
With hypotheses as in Theorem \ref{thm_gen},
we have that on $\good(d)$ the event $\good_\pm^{\sparse}(\eta)$ holds 
with probability $1-O(n^{-100})$ if 
$\eta\in [C_1'd^{-1/4}, 1]$ for a sufficiently large absolute constant $C_1'>0$.
\end{proposition}

As for the rrd matrix $M$ and Theorem \ref{thm_main}, we will also treat structured null vectors separately, but it turns out that sparsity is no longer the right notion of structure. Instead, we will need to show that null vectors of $M$ have \emph{small level sets}:

\begin{proposition}[No structured null vectors for $M$]
\label{prop_struct}
For $\eta\in (0,1]$, let $\good^{\sls}(\eta)$ be the event that for any nontrivial left or right null vector $x$ of $M$ and for any $\lambda\in \R$,
\begin{equation}
|x^{-1}(\lambda)| \le \eta n.
\end{equation}
With hypotheses as in Theorem \ref{thm_main}, 
we have that $\good^{\sls}(\eta)$ holds 
with probability $1-O(n^{-100})$ if 
$\eta \ge C_1d^{-c_0}$ for some absolute constants $C_1,c_0>0$ sufficiently large and small, respectively (the constant $c_0$ is the same as in Theorem \ref{thm_main} and can be taken to be $1/18$).
\end{proposition}

We prove 
Propositions \ref{prop_structpm} and \ref{prop_struct} 
in Section \ref{sec_struct}.
The reason for ruling out null vectors with large level sets will be apparent in the next section, where we describe our couplings approach.

\subsection{Injecting a random walk}
\label{sec_inject0}

The proof in Section \ref{sec_komlos} proceeded by reducing to the event that $R_1\cdot u=0$, where $R_1$ is the first row of $\Xi$ and $u$ is a unit vector in $V_1^\perp=\Span(R_2,\dots,R_n)^\perp$.
Then we used independence of the entries of $\Xi$ in two ways:
\begin{enumerate}
\item Independence of the rows of $\Xi$ allowed us to condition on $R_2,\dots, R_n$ to fix $u$, without affecting the distribution of $R_1$.
\item Independence of the components of $R_1$ allowed us to view the dot product $R_1\cdot u$ as a random walk, to which we could apply the anti-concentration result Theorem \ref{thm_erdos}.
\end{enumerate}
The rrd matrix $M$ enjoys neither of these properties. 
However, we will be able to accomplish something like (2) above by defining an appropriate coupling of rrd matrices using switchings. 
It will take some care to implement this without having the independence between rows (1).

To illustrate our couplings approach, let us consider a toy problem: to control the event that the first two rows lie in the span of the remaining rows, i.e. to show 
\begin{equation}		\label{toy0}
\pr\set{ R_1,R_2\in V_{(1,2)}^\perp} = o(1)
\end{equation}
where $V_{(1,2)}:= \Span(R_3,\dots, R_n)$. 
We will see later that this can be used to control the event that $M$ has corank at least 2 (see Lemma \ref{lem_sampling}).
For now we will operate under the following
\begin{assumption}\label{ass_dense}
$n\ll d \le \frac{n}2$.
\end{assumption}
Thus, we are assuming $M$ is a dense rrd matrix.
In the next section we will discuss some of the new ideas necessary to treat sparse matrices.

Blindly following the proof from Section \ref{sec_komlos}, we condition on the rows $R_3,\dots, R_n$ to fix the space $V_{(1,2)}$, and pick a unit vector $u\in V_{(1,2)}^\perp$, say uniformly and independently of $R_1,R_2$ under the conditioning.
Now it suffices to show
\begin{equation}		\label{toy1}
\pr\big( R_1\cdot u=0 \,\big|\,R_3,\dots,R_n\big) = o(1).
\end{equation}

We need to understand how $R_1$ and $R_2$ are distributed under the conditioning on $R_3,\dots, R_n$. 
Recall from \eqref{codef}, \eqref{exdef} the sets
$\Co_M(1,2)$, $\Ex_M(1,2)$, $\Ex_M(2,1)$, which the partition the vertex-pair neighborhood $\mN_M(\set{1,2})$.
Now since the entries of each column sum to $d$, by fixing $R_3,\dots, R_n$
we have fixed which columns of $M$ need both, neither, or just one of their first two components equal to 1 in order to meet the constraint.
This fixes the sets $\Co_M(1,2)$ and $\Ex_M(1,2)\cup \Ex_M(2,1)$. 
Furthermore, by the row sums constraint, we must have 
\begin{align}
|\Ex_M(1,2)|&=d-|\Co_M(1,2)|	=|\Ex_M(2,1)| 	\label{exco}
\end{align}
It follows that with $R_3,\dots, R_n$ fixed, the only remaining randomness is in the uniform random equipartition of the deterministic set  $\Ex_M(1,2)\cup \Ex_M(2,1)$ into the sets $\Ex_M(1,2)$, $\Ex_M(2,1)$.
See Figure \ref{fig_shuffling}.

\begin{figure}
\centering
  \includegraphics[width=90mm]{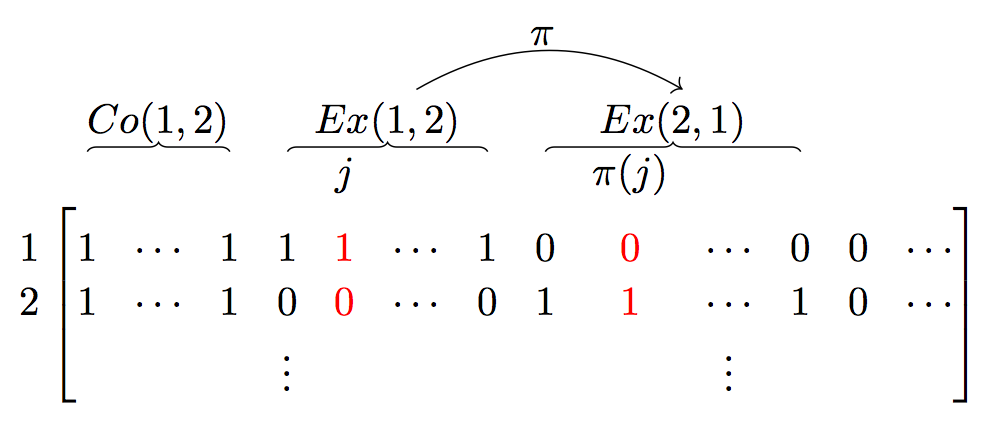}

\caption{The shuffling coupling: the minor $M_{(1,2)\times(j,\pi(j))}$ (in red) will be replaced with $\eye$ or $\jay$ according to a random sign $\xi(j)$. We do this independently for each $j\in \Ex(1,2)$.}

\label{fig_shuffling}
\end{figure}

We re-randomize the sets $\Ex_M(1,2)$, $\Ex_M(2,1)$ in the following way.
Under this conditioning, pick a bijection 
$\pi: \Ex_M(1,2)\rightarrow \Ex_M(2,1)$
uniformly at random. 
Now for each $j\in \Ex_M(1,2)$ we have
$$M_{(1,2)\times(j,\pi(j))}= \idd=: \eye.$$
Having obtained a sequence of ``switchable" $2\times 2$ minors,
we can apply random switchings (with terminology as in Section \ref{sec_strategy}).
Let $\xi:[n]\rightarrow \set{\pm1}$ be a sequence of iid uniform signs, independent of all other variables.
For each $j\in \Ex_M(1,2)$, we replace the minor $M_{(1,2)\times(j,\pi(j))}$ with the random minor
$$\eye \un(\xi(j)=+1)+\jay \un(\xi(j)=-1)$$
(with notation as in \eqref{IJ}).
Call the resulting matrix $\tM$. 
It is not hard to show that $\tM$ is also an rrd matrix after undoing all of the conditioning (see the proof of Lemma \ref{lem_shuffling} below).

We have hence obtained a coupled pair $(M,\tM)$ of rrd matrices (more precisely, we have defined a coupling $(M,\tM,\pi,\xi)$ on an enlarged probability space such that the marginals for the first two entries are uniform). 
Let $\tR_i$ denote the $i$th row of $\tM$. 
Replacing $M$ with $\tM$ in \eqref{toy1}, it suffices to show
\begin{equation}
\pr\big(\tR_1\cdot u =0 \, \big|\, M\big) =o(1).
\end{equation}
Now in the randomness of the iid signs $\xi(j)$, one sees that the dot product $\tR_1\cdot u$ is a random walk:
\begin{equation}	\label{toyrw}
\tR_1\cdot u = W_0 + \sum_{j\in \Ex_M(1,2)} \xi(j) \frac{u(j)-u(\pi(j))}{2}
\end{equation}
where $W_0$ is a term that does not depend on $\pi$ or $\xi$. 
Applying Theorem \ref{thm_erdos} we have
$$\pr_{\xi}\big\{\tR_1\cdot u =0 \big\} \ll \big|\big\{ j\in \Ex_M(1,2): \, u(j)\ne u(\pi(j))\big\} \big|^{-1/2}.$$
It remains to get a lower bound on the number of $j\in \Ex_M(1,2)$ for which $u(j)\ne u(\pi(j))$. 

First we deal with the possibility that $\Ex_M(1,2)$ is a very small set. 
On average, we expect $\Co_M(1,2)$ to be of size roughly $p^2n=d^2/n$. 
By \eqref{exco} and our assumption $p\le 1/2$ (from Remark \ref{rmk_assumed}) we have 
$\e\big|\Ex_M(1,2)\big|\gg d.$
It was shown in \cite{Cook:discrep} that codegrees in random regular digraphs are sharply concentrated (see Theorem \ref{thm_codeg}) from which we can deduce that
\begin{equation}	\label{toyex}
\big|\Ex_M(1,2)\big|\gg d
\end{equation}
off a negligibly small event.

Now we apply Proposition \ref{prop_struct} and the randomness of $\pi$ to argue that for most $j\in \Ex_M(1,2)$ we have $u(j)\ne u(\pi(j))$.
Since $u\in V_{(1,2)}^\perp$ we have that $u$ is a (right) null vector of the $(n-2)\times n$ matrix $M_{[3,n]\times [n]}$.
By a small extension of Proposition \ref{prop_struct} we may assume that $u$ is unstructured, i.e.\ that all of its level sets are of size at most $\eta n$, with $\eta$ of size $\Theta(d^{-c_0})$ for some $c_0>0$ absolute. 
(In the actual proof we will argue that $u$ is unstructured in a slightly different way, but in any case it comes down to an application of Proposition \ref{prop_struct}.)
Now by \eqref{toyex} and Assumption \ref{ass_dense}, the sets $\Ex_M(1,2)$, $\Ex_M(2,1)$ are much larger than the level sets of $u$.
Hence, in the randomness of $\pi$, it is very unlikely that we have $\pi(j)\in u^{-1}(u(j))$ for a large number of indices $j\in \Ex_M(1,2)$. 
Thus, off a negligibly small event we can deduce that most of the steps taken by the random walk \eqref{toyrw} are nonzero, and hence
\begin{equation}	\label{toypixi}
\pr_{\xi, \pi}\big\{\tR_1\cdot u =0 \big\} \ll \big|\Ex_M(1,2)\big|^{-1/2} \ll d^{-1/2}.
\end{equation}
Since we are assuming $d=\omega(1)$, we have completed the proof of \eqref{toy0}.

To summarize, we bounded $\pr(R_1,R_2\in V_{(1,2)}^\perp)$ by defining a coupling $(M,\tM,\pi,\xi)$ on an enlarged probability space, with $\tM\eqd M$, and replacing $M$ with $\tM$ in \eqref{toy1}.
The variables $M$, $\pi$ and $\xi$ each played a special role:
\begin{enumerate}
\item In the randomness of $M$, we simply restricted to a couple of ``good events": the event $\good^{\sls}(\eta)$ that null vectors are unstructured, and the event that codegrees are close to their expectations.
\item Conditional on $M$ satisfying the good events, $\pi$ was used to pair indices in $\Ex_M(1,2)$ with indices in $\Ex_M(2,1)$ to show that, off a small event, the random walk $\tR_1\cdot u$ takes many nonzero steps $\frac12(u(j)-u(\pi(j)))$.
\item Conditional on good realizations of $M$ and $\pi$, the randomness of $\xi$ was used with Theorem \ref{thm_erdos} to finish the proof.
\end{enumerate}

As remarked above, \eqref{toy0} can be used to deduce that $M$ has corank at most 1 
a.a.s.
It then remains to deal with the event that $\corank(M)=1$.
This task is a little more complicated, and involves expressing a certain $2\times 2$ determinant involving two randomly sampled rows of $M$ as a random walk.
See Section \ref{sec_corank1} for details. 

\subsection{Dealing with sparsity}
\label{sec_sparsity}

In the previous section, we used Assumption \ref{ass_dense} to guarantee that the level sets of the normal vector $u$ were small in comparison to the neighborhood $\mN_M(\set{1,2})$ (more precisely, the sets $\Ex_M(1,2)$ and $\Ex_M(2,1)$). 
Indeed, since the level sets are of size at most $\eta n =\Theta(nd^{-c_0})$ by Proposition \ref{prop_struct}, and since $|\Ex_M(1,2)|\gg d$ with high probability, we see upon rearranging that in the above argument we must assume $d\ge Cn^{1/(1+c_0)}$ for a sufficiently large constant $C>0$.
It turns out that the value $c_0=1/8$ is the limit of what can be obtained by our arguments in the proof of Proposition \ref{prop_struct} for the case that $d=\omega(\sqrt{n})$.
Hence, the argument of the previous section is limited to $d\ge Cn^{8/9}$.

In the present work we are able to take $d$ as small as 
$C_0\log^2n$ using some new ideas. 
Rather than consider the event that $R_1,R_2\in V_{(1,2)}^\perp$, we will draw row indices $I_1,I_2$ \emph{at random} and seek to bound 
$
\pr\big\{\, R_{I_1}, R_{I_2}\in V_{(I_1,I_2)}^\perp \, \big\}.
$
It can be shown that this leads to control on the event that $\corank(M)\ge 2$ (see Lemma \ref{lem_sampling}).
Conditional on $I_1,I_2$ and the remaining rows $(R_i)_{i\notin \set{I_1,I_2}}$, we will again select a unit normal vector $u$ uniformly at random.

Whereas in Section \ref{sec_inject0} the distribution of $u$ played no special role, here we will use it along with the randomness of $I_1,I_2$ to argue that it is very unlikely that a level set of $u$ has large overlap with the neighborhood $\mN_M(\set{I_1,I_2})$.
Under conditioning on $I_1,I_2$, one can see that the ``bad" realizations of $u$ form an algebraic subset of the sphere.
We will then use the simple fact that a proper algebraic subset of the sphere has surface measure zero.
(This is perhaps the only part of the proof that is not strictly combinatorial.)
The argument requires some care as the vector $u$ and the set $\mN_M(\set{I_1,I_2})$ are both dependent on $I_1,I_2$.
See Section \ref{sec_corank2} for the detailed proof.

\begin{remark}[Necessary lower bounds on $d$]		\label{rmk_dlb}
While we need to assume $\min(d,n-d)\ge C_0\log^2n$ in Theorem \ref{thm_main}, various parts of the argument work under a weaker lower bound assumption.
Specifically, Theorem \ref{thm_main} follows from Proposition \ref{prop_red} and our ability to restrict to the following ``good events":
\begin{enumerate}
\item $\good^{\sls}(\eta)$ (from Proposition \ref{prop_struct}), with $\eta$ of order $d^{-c_0}$;
\item $\good^{\ex}(\delta)$ (from Theorem \ref{thm_codeg}), with $\delta$ a small fixed constant.
\end{enumerate}
Proposition \ref{prop_struct} shows 1.\ holds with high probability under the hypothesis $\min(d,n-d)\ge C_0\log^2n$ for $C_0>0$ sufficiently large.
Theorem \ref{thm_codeg} establishes 2.\ holding with overwhelming probability if $\min(d,n-d)=\omega(\log n)$ (and in fact holds with high probability if $d\ge C_0\log n$). 
Finally, Proposition \ref{prop_red} itself assumes no lower bound on $d$.
Hence, the only real barrier to assuming a lower bound of order $\log n$ is Proposition \ref{prop_struct}. The lower bound $\min(d,n-d)\ge C_0\log^2n$ is only needed there for technical reasons, and we believe that an improvement to $C_0\log n$ is possible. 
Beyond that, it is likely that an entirely different approach will be needed for the case $\min(d,n-d)=o(\log n)$, as the discrepancy properties in Section \ref{sec_discrepancy} would no longer hold with high probability, and these are essential to several parts of our argument.

Similar comments apply to Theorem \ref{thm_gen}, where the lower bound assumption $d\ge C_0'\log^2n$ comes from Proposition \ref{prop_structpm}.
\end{remark}

\begin{remark}\label{rmk:LLKT-JY2}
As was mentioned in \Cref{rmk:LLKT-JY}, an extension of Theorem \ref{thm_main} to the range $C\le d\le cn/\log^2n$ has recently been accomplished in \cite{LLKT-JY}, for some absolute constants $C,c>0$.
The argument in \cite{LLKT-JY} builds on the approach of the present work, and is similar in its use of a shuffling coupling (much like Lemma \ref{lem_shuffling}) and graph discrepancy properties. 
To take $d$ below the barrier $\log n$ discussed in Remark \ref{rmk_dlb}, they are able to make use of weaker discrepancy properties than the ones employed in the present work.
Another notable difference from the present work is that they are able to effectively apply the shuffling coupling with much less control on ``structured null vectors" than is provided by Proposition \ref{prop_struct}.
\end{remark}

\section{Preliminaries}
\label{sec_prelim}

\subsection{The shuffling coupling}
\label{sec_shuffling}

In this section we formally define the pair $(M,\tM)$ of rrd matrices described in the previous section, where $\tM$ is obtained by re-randomizing the neighborhood $\mN_M(\set{i_1,i_2})$ of a pair of distinct vertices $i_1,i_2\in [n]$ in a certain way. 
Recall that from the row sums constraint we have 
$$
\ex_M(i_1,i_2) = d-\co_M(i_1,i_2) = \ex_M(i_2,i_1)
$$
for any distinct $i_1,i_2\in [n]$ (recall from Section \ref{sec_notation} our notation $\ex_M(i_1,i_2)=|\Ex_M(i_1,i_2)|$, $\co_M(i_1,i_2)=|\Co_M(i_1,i_2)|$).
On an intuitive level, the shuffling operation is somewhat similar to performing a ``riffle shuffling" of the ``deck" $\Ex_M(i_1,i_2)\cup \Ex_M(i_2,i_1)$, then cutting the deck into two equal parts to obtain $\Ex_{\tM}(i_1,i_2)$, $\Ex_{\tM}(i_2,i_1)$.
The set of common neighbors $\Co_M(i_1,i_2) = \mN_M(i_1)\cap \mN_M(i_2)$ is preserved by the shuffling.

\begin{definition}[Shuffling]
\label{def_shuffling}
Let $M\in \mM_{n,d}$ and $i_1,i_2\in [n]$ distinct. 
For a bijection
$$\pi:\Ex_M(i_1,i_2)\rightarrow \Ex_M(i_2,i_1)$$
and a sequence of signs $\xi:[n]\rightarrow\set{\pm1}$, 
by \emph{perform a shuffling on $M$ at rows $(i_1,i_2)$ according to $\pi,$ $\xi$,} 
we mean to replace the $2\times 2$ minors $M_{(i_1,i_2)\times(j,\pi(j))}$ with 
$$\eye \un(\xi(j)=+1) + \jay \un(\xi(j)=-1)$$
for each $j\in \Ex_M(i_1,i_2)$, and to leave all other entries of $M$ unchanged. 
\end{definition}

The key to applying the shuffling operation in the proof of Theorem \ref{thm_main} will be to take $\pi$ and $\xi$ to be random. 

\begin{lemma}[Shuffling coupling]
\label{lem_shuffling}
Let $M$ be an rrd matrix, and fix $i_1,i_2\in [n]$ distinct. 
Conditional on $M$, let $\pi:\Ex_M(i_1,i_2)\rightarrow \Ex_M(i_2,i_1)$ be a uniform random bijection. 
Draw a sequence $\xi:[n]\rightarrow \set{\pm 1}$ of iid uniform signs, independent of all other variables. 
Form $\tM$ by performing a shuffling on $M$ at rows $(i_1,i_2)$ according to $\pi$ and $\xi$. 
Then $\tM\eqd M$. 
\end{lemma}

At one part of the proof we will need the following slightly more general version (which implies the above lemma)
in which there is a fixed set of ``frozen" columns which we cannot modify.
For a set $A$ and an integer $0\le k\le |A|$, we use the notation ${A\choose k}$ for the set of subsets of $A$ of size $k$.

\begin{lemma}[Restricted shuffling]
\label{lem_rshuffling}
Let $M$ be an rrd matrix and fix $i_1,i_2\in [n]$ distinct. 
Let $\Fix\subset[n]$ be a set of column indices that is fixed by conditioning on the rows $(R_i)_{i\notin \set{i_1,i_2}}$.
Set 
\begin{equation}
A_1= \Ex_M(i_1,i_2)\setminus \Fix, \quad A_2 = \Ex_M(i_2,i_1)\setminus \Fix
\end{equation}
and let $s\le \min(|A_1|,|A_2|)$ also be fixed by conditioning on the rows $(R_i)_{i\notin\set{i_1,i_2}}$ (i.e.\ chosen measurably with respect to the sigma algebra generated by these rows).
Conditional on $M$ let 
$S_1\in {A_1\choose s}$ and $S_2\in {A_2\choose s}$
be chosen independently and uniformly.
Conditional on $M, S_1,S_2$, let
$
\pi: S_1\rightarrow S_2
$
be a uniform random bijection.
Finally, let $\xi:[n]\rightarrow\set{\pm1}$ be a sequence of iid uniform signs, independent of all other variables.  

Form $\tM$ from $M$ by replacing the $2\times 2$ minors $ M_{(i_1,i_2)\times(j,\pi(j))}$ with
$$\eye \un(\xi(j)=+1) + \jay \un(\xi(j)=-1)$$
for each $j\in S_1$, leaving all other entries of $M$ unchanged. 
Then $\tM\eqd M$.
\end{lemma}

Lemma \ref{lem_shuffling} follows from Lemma \ref{lem_rshuffling} by taking $\Fix$ to be empty and $s=\ex_M(i_1,i_2)$. 

\begin{proof}
Condition on the rows $(R_i)_{i\notin \set{i_1,i_2}}$. 
This fixes $\Fix$ and the set
$\Ex_M(i_1,i_2)\cup \Ex_M(i_2,i_1).$
Condition also on the columns of $M$ with indices in $\Fix$ -- this fixes $a_1:=|A_1|$ and $a_2:=|A_2|$. 

The only remaining randomness of $M$ is in the uniform random partition of 
$$A:=\big(\Ex_M(i_1,i_2)\cup \Ex_M(i_2,i_1)\big)\setminus \Fix$$
into the sets $A_1,A_2$ of prescribed sizes.
It hence suffices to show that $\tA_1:=\Ex_{\tM}(i_1,i_2)\setminus \Fix$ is also distributed uniformly over ${A\choose a_1}$.
We may write $\tA_1$ as the disjoint union
\begin{equation}	\label{ta1def}
\tA_1 = (A_1\setminus S_1)\, \sqcup\, \big(S_1\cap\xi^{-1}(+1)\big)\, \sqcup\, \pi\big(S_1\cap \xi^{-1}(-1)\big).
\end{equation}

To see that $\tA_1\eqd A_1$ it is clearer to use the following alternative description of the coupling $(M,\tM, S_1,S_2,\pi, \xi)$. 
Denote
\begin{align*}
E_1 &:= [a_1] \supset[s]=: F_1 \\
E_2 &:= [a_1+1,a_1+a_2] \supset [a_1+1,a_1+s]=: F_2
\end{align*}
and denote $E=E_1\cup E_2$, $F=F_1\cup F_2$. 
Under the above conditioning, draw bijections
\begin{equation}
\Phi: E\rightarrow A, \quad \check{\pi}: F_1\rightarrow F_2
\end{equation}
independently and uniformly at random, and let $\check{\xi}:[n]\rightarrow\set{\pm1}$ be a sequence of iid uniform signs independent of all other variables. 
Then
$$(A_1,A_2,S_1,S_2,\pi, \xi)\eqd \big(\Phi(E_1),\Phi(E_2),\Phi(F_1),\Phi(F_2),\Phi\circ\check{\pi}\circ\Phi^{-1}, \check{\xi}\circ\Phi^{-1}\big).$$
We have shifted the randomness of the sets $A_1,A_2,S_1,S_2$ to 
the randomness of the map $\Phi$.
We want to show that $\tA_1\eqd \Phi(E_1)$, where $E_1=[a_1]$ is now a deterministic set.
From \eqref{ta1def}, 
\begin{align*}
\tA_1 &\eqd \Phi(E_1\setminus F_1) \, \sqcup \, \Phi\big(F_1\cap \check{\xi}^{-1}(+1)\big) \, \sqcup\, (\Phi\circ \check{\pi})\big(F_1\cap\check{\xi}^{-1}(-1)\big) \\
& = \Phi\Big[ (E_1\setminus F_1) \, \sqcup \, \big(F_1\cap \check{\xi}^{-1}(+1)\big) \, \sqcup\, \check{\pi}\big(F_1\cap\check{\xi}^{-1}(-1)\big) \Big].
\end{align*}
Conditioning on $\check{\pi}$ and $\check{\xi}$ (which doesn't affect the distribution of $\Phi$) we have that $\tA_1$ is the image under $\Phi$ of a fixed set of size $a_1$, which completes the proof.
\end{proof}

\subsection{Discrepancy properties}
\label{sec_discrepancy}

In this section we collect various ``good events" concerning the distribution of edges in the random regular digraph $\Gamma$ associated to $M$. 
In all cases, the good event is shown to hold with overwhelming probability, for a suitable range of parameters and assuming $\min(d,n-d)=\omega\big(\log n\big)$
(note this is a wider range of $d$ than is assumed in Theorems \ref{thm_main} and \ref{thm_gen}).
This will allow us to restrict to these events without further comment in subsequent stages of the proof.
(Indeed, note that we are ultimately aiming for only a polynomially-small bound on the singularity probability, so the failure probabilities for the good events will be negligible.)
At the end of the section we prove that Theorem \ref{thm_pm} follows from Theorem \ref{thm_gen} by showing that the event $\good(d)$ from the latter theorem holds with overwhelming probability for $M$.

The results of this section are all corollaries of sharp tail estimates for codegrees and edge counts in random regular digraphs (Theorems \ref{thm_codeg} and \ref{thm_edge} below).
The proofs, which are too long for inclusion in the present work, are contained in the companion paper \cite{Cook:discrep}.
These results may also be of independent interest for graph theorists.

The shuffling coupling from Lemma \ref{lem_shuffling} will only be useful if the sets $\Ex_M(i_1,i_2)$ are large (see Section \ref{sec_inject0}).
Hence, the following result from \cite{Cook:discrep} will be essential for our arguments.
Recall that $p:=d/n$ denotes the average edge density for the digraph.

\begin{theorem}[Concentration of codegrees \cite{Cook:discrep}]
\label{thm_codeg}
For $\delta\in (0,1)$, let $\good^{\ex}(\delta)$ denote the event that for every pair of distinct $i_1,i_2\in [n]$ we have
\begin{equation}
 \left| \frac{\ex_M(i_1,i_2)}{p(1-p)n} - 1\right| \le \delta \quad \mbox{ and } \quad  \left| \frac{\ex_{M^\tran}(i_1,i_2)}{ p(1-p)n} -1 \right| \le \delta.
\end{equation}
Then
\begin{equation}
\pro{\good^{\ex}(\delta)} = 1 - n^{O(1)}\expo{ - c\delta \min \big\{d,n-d, \delta n\big\}}.
\end{equation}
In particular, for any fixed $\delta\in(0,1)$ independent of $n$ we have that $\good^{\ex}(\delta)$ holds with overwhelming probability if $\min(d,n-d)=\omega(\log n)$.
\end{theorem}

Our next result concerns the concentration of the number of edges $e_M(A,B)$ passing from a set $A$ to a set $B$ (defined in \eqref{defedge}).
We expect this random variable to be of size roughly $\mu(A,B):=p|A||B|$.
It is straightforward to check that from the $d$-regularity constraint, for any $t\in \R$ we have the following equality of events:
\begin{equation}
\big\{ e_M(A,B) - p|A||B| \ge t \big\} = \big\{ e_M(A^c,B^c) -p|A^c||B^c| \ge t \big\}
\end{equation}
where we denote $A^c=[n]\setminus A$.
That is, a large deviation of $e_M(A,B)$ coincides with a large deviation of $e_M(A^c,B^c)$. 
It will hence be natural to express deviations of $e_M(A,B)$ at the scale 
\begin{equation}
\hat{\mu}(A,B) := p \min\big\{ |A||B|, (n-|A|)(n-|B|) \big\}.
\end{equation}

\begin{theorem}[Concentration of edge counts \cite{Cook:discrep}]
\label{thm_edge}
With $\good^{\ex}(\delta)$ as in Theorem \ref{thm_codeg}, we have that for any $A,B\subset [n]$ and any $\tau\ge 0$, 
\begin{equation}
\pro{ \Big\{ \big| e_M(A,B) - \mu(A,B) \big| \ge \tau \hat{\mu}(A,B) \Big\} \wedge \good^{\ex}(\delta) } 
\le 2\expo{ -\frac{c\tau^2}{1+\tau} \hat{\mu}(A,B) }
\end{equation}
provided $\delta \le \min\big( \frac{1}{4}, \frac{\tau}{8}\big)$. 
\end{theorem}

Combining the above theorems with a union bound over pairs of vertex sets $(A,B)$,
we can deduce that with overwhelming probability, the densities of all sufficiently large minors of $M$ are uniformly close to their expectations.
The following is a consequence of Corollary 1.8 in \cite{Cook:discrep} for the case that $\min(d,n-d)=\omega(\log n)$ -- see \cite{Cook:discrep} for the bound with explicit dependence on $n$, $d$ and the parameter $\eps$.

\begin{corollary}[Discrepancy for large minors \cite{Cook:discrep}]
\label{cor_discrep}
Let 
$C>0$ be a sufficiently large absolute constant. 
For $\eps\in (0,1)$, define the family of pairs of sets
\begin{equation}
\mF(\eps)= \set{(A,B):\; A, B\subset[n],\; \min(|A|,|B|) \ge \frac{C}{\eps^2}\frac{\log n}{p} }
\end{equation}
and the event
$$\good^{\ee}(\eps) = \Big\{ \forall (A,B)\in \mF(\eps), \;\big|e_M(A,B) - \mu(A,B) \big| \le \eps \hat{\mu}(A,B) \Big\}. $$
If $\min(d,n-d)=\omega(\log n)$ and $\eps\in (0,1)$ is fixed independent of $n$, then $\good^{\ee}(\eps)$ holds with overwhelming probability.
\end{corollary}

Note that for $S,B\subset[n]$ we have the deterministic bound
\begin{equation}	\label{dsbound}
e_M(S,B) \le d|S|
\end{equation}
which is effective when $S$ is small (and we have equality when $B=[n]$). 
While this bound will be sufficient for many purposes, we will sometimes need a little more when $|B|=o(n)$.
Theorem \ref{thm_edge} allows us to improve on \eqref{dsbound} off a small event:

\begin{corollary}[Discrepancy for thin minors]
\label{cor_thin}
For $\gamma >0$ set
\begin{equation}	\label{s0}
s_0(\gamma):= \frac{\log n}{2\gamma }\frac{n}{d}
\end{equation}
and for $\eps_0\in (0,1]$ define the family of ``thin minors"
\begin{equation}
\mF_{\thin}(\eps_0,\gamma) = \Big\{ (S,B) : |S|\le s_0(\gamma),  
\; |B| \le \frac{\eps_0 \gamma}{\log n}d|S| \Big\}.
\end{equation}
Let
\begin{equation}	\label{bepsg}
\bad(\eps_0,\gamma) = \Big\{ \exists\, (S,B)\in \mF_{\thin}(\eps_0,\gamma): \; 
\max\big(e_M(S,B) , e_M(B,S)\big) \ge \eps_0 d|S|\Big\}.
\end{equation}
There are absolute constants 
$C,c>0$ such that if $\min(d,n-d)\ge C\eps_0^{-1}\log n$ and $\gamma\in (0,c]$, then
\begin{equation}
\pro{ \bad(\eps_0, \gamma) }\ll \expo{-c\eps_0 d}.
\end{equation}
In particular, if $\min(d,n-d)=\omega(\log n)$ and $\eps_0$ is fixed independent of $n$, then $\bad(\eps_0,\gamma)^c$ holds with overwhelming probability.
\end{corollary}

\begin{proof}
Let $\eps_0\in (0,1]$ and $\gamma>0$.
Denote
\begin{equation}	\label{b0}
b_0(\eps_0,\gamma, s) :=  \frac{\eps_0 \gamma}{\log n}ds
\end{equation}
and
fix $s\le s_0(\gamma)$, $b\le b_0(\eps_0,\gamma,s)$. 
For $S,B\subset [n]$ with $|S|=s$, $|B|=b$ we have
\begin{equation}	\label{mubound}
\mu(S, B) = \frac{dsb}{n} \le \frac{b_0ds}{n}  \le \frac{\eps_0 ds}2.
\end{equation}
Let $\delta>0$ to be chosen. For fixed $S, B$ as above,
from Theorem \ref{thm_edge} we have
\begin{align}
\pro{ \good^{\ex}(\delta)\wedge \Big\{ e_M(S,B) \ge \eps_0 ds \Big\}} 
&\le \pro{ \good^{\ex}(\delta)\wedge \Big\{ e_M(S,B) - \mu(S,B) \ge \frac{\eps_0}{2} ds \Big\}}  \notag\\
&\le \expo{ -c\eps_0 ds}	\label{bsfixed}
\end{align}
provided we take 
$\delta \le \min\left(\frac14, \frac{\eps_0 d s}{16 \mu}\right).$
Since $\frac{\eps_0 d s}{16\mu} \ge \frac18$ by \eqref{mubound}, we can take $\delta = \frac18$. 
With this choice of $\delta$ we have
\begin{equation}	\label{gexbound}
\pro{\good^{\ex}(\delta)^c} \ll n^{O(1)} \expo{-cd}
\end{equation}
from Theorem \ref{thm_codeg}.
Now by a union bound, \eqref{bsfixed} and the assumed lower bound on $d$, 
\begin{align*}
\pro{ \good^{\ex}(1/8) \wedge\bad(\eps_0,\gamma) } 
&\le \sum_{s\le s_0(\gamma)} {n\choose s} \sum_{b\le b_0(\eps_0,\gamma,s)} {n\choose b} \expo{-c\eps_0 ds} \\
&\ll \sum_{s\le s_0(\gamma)} n^{b_0} \expo{s(\log n - c\eps_0 d)} \\
&\le \sum_{s\le s_0(\gamma)} \expo{ b_0\log n-c\eps_0 ds} \\
&\le \sum_{s\le s_0(\gamma)} \expo{-c\eps_0 ds} \\
&\ll \expo{-c\eps_0 d}
\end{align*}
where in the fourth line we used the definition \eqref{b0} of $b_0$ and took $\gamma$ sufficiently small. 
Combining with the bound \eqref{gexbound} and the lower bound on $d$ completes the proof.
\end{proof}

We have the following quick consequence that with high probability, the size of the neighborhood $\mN_M(S)$ of any small set $S$ is within a logarithmic factor of the upper bound $d|S|$. 

\begin{corollary}[Expansion of small sets]
\label{cor_expand}
For $\gamma>0$, let $\good^{\exp}(\gamma)$ be the event that for every $S\subset[n]$ with $|S|\le s_0(\gamma)= \frac{\log n}{2\gamma}\frac{n}{d}$, we have
$$\big|\mN_M(S) \big| \ge \frac{\gamma}{\log n} d|S|.$$
Assume $d=\omega(\log n)$.
Then there is a constant $c_1>0$ such that for all $\gamma\in (0,c_1]$, 
\begin{equation}
\pro{\good^{\exp}(\gamma)} = 1- O\big( e^{-cd} \big).
\end{equation}
\end{corollary}

\begin{proof}
Let $\gamma>0$. 
From the crude lower bound $|\mN_M(S)|\ge d$ we have that $\good^{\exp}(\gamma)$ trivially holds if $d\ge n/2$, so assume $\omega(\log n)\le d\le n/2$.
On $\good^{\exp}(\gamma)^c$ there exist $S,B\subset [n]$ with $|S|\le s_0(\gamma)$ and 
$$|B|< \frac{\gamma}{\log n} d|S| = b_0(1,\gamma,|S|)$$
in the notation of \eqref{s0}, \eqref{b0}, such that $e_M(S,B) = d|S|$ (simply from taking $B=\mN_M(S)$). 
Hence, $\good^{\exp}(\gamma)^c$ is contained in the event $\bad(1,\gamma)$ from Corollary \ref{cor_thin}, and the result follows by taking $\gamma$ sufficiently small.
\end{proof}

Now we can prove that Theorem \ref{thm_pm} follows from Theorem \ref{thm_gen}.

\begin{proof}{\emph{of Theorem \ref{thm_pm}}}.
Assume $\omega(\log n)\le d\le n$. 
Write
$M_\pm=M\schur \Xi$
as in \eqref{unsigned}, where $M$ is an rrd matrix and $\Xi$ is matrix of iid uniform signs, independent of $M$.
It suffices to show that the event $\good(d)$ in Theorem \ref{thm_gen} holds with overwhelming probability for $\Sigma=M$.

Conditions (0) and (3) of $\good(d)$ are immediate for $M$ (and hold with probability 1), 
taking $\kappa_3=1$.
From Corollary \ref{cor_expand} we have that condition (1) is satisfied with probability $1-O(e^{-cd})$. 
From Corollary \ref{cor_discrep}, we have that if $\min(d,n-d)=\omega(\log n)$, then condition (2) holds with overwhelming probability with $c_2\in (0,1)$ fixed arbitrarily and $C_2>0$ sufficiently large depending only on $c_2$.

It only remains to show that condition (2) holds with overwhelming probability for the high density case $d=n-O(\log n)$. 
Let $A,B\subset [n]$ such that 
$|A|,|B| \ge C_2\frac{n}{d}\log n$
for some $C_2>0$ to be chosen sufficiently large. 
Let $M'$ denote the complementary rrd matrix with entries $M'(i,j)=1-M(i,j)$. 
Then by $(n-d)$-regularity,
$e_{M'}(A,B)\le (n-d)|A|.$
It follows that
\begin{align*}
e_M(A,B) &= |A||B| - e_{M'}(A,B) \\
&\ge |A|(|B|-(n-d)) \\
&\gg |A||B|\\
&\ge \frac{d}{n}|A||B|
\end{align*}
where in the third line we used the upper bound $n-d=O(\log n)$ and the lower bound on $|B|$, taking $C_2$ sufficiently large. 
It follows that for the case $d=n-O(\log n)$, we may take $c_2\in (0,1)$ sufficiently small such that condition (2) holds with probability 1 for all $n$ sufficiently large.
\end{proof}

\subsection{Concentration of measure}
\label{sec_concentration}

The following concentration inequality for certain functions on the symmetric group will be useful when working with the bijections $\pi$ in the shuffling coupling of Lemma \ref{lem_shuffling}, and follows from the $d=1$ case of Theorem 1.18 in \cite{Cook:discrep}, or alternatively from Proposition 1.1 in \cite{Chatterjee}.

\begin{lemma}[Concentration for the symmetric group]
\label{lem_com}
For $m\ge 1$,  $\pi \in \sym(m)$ a permutation on $[m]$, and $A,B\subset[m]$, denote
$$e_\pi(A,B) = \big|\big\{ i\in A: \pi(i)\in B\big\}\big|.$$
If $\pi$ is a uniform random element of $\sym(m)$, we have that for any $\tau\ge0$,
\begin{equation}
\pr\Big\{ \big| e_\pi(A,B) - |A||B|/m \big| \ge \tau |A||B|/m\Big\} \le 2\exp\bigg\{ -\frac{c\tau^2}{1+\tau} \frac{|A||B|}{m}\bigg\}.
\end{equation}
\end{lemma}

\begin{remark}
Note that the above lemma is essentially the $d=1$ case of Theorem \ref{thm_edge}, the only difference (apart from constants in the exponential) being that we do not need to restrict to any ``good event" like $\good^{\ex}(\delta)$.
\end{remark}

\section{Unstructured null vectors}
\label{sec_unstruct}

In this section we prove Theorems \ref{thm_main} and \ref{thm_gen},
taking as black boxes Propositions \ref{prop_struct} and \ref{prop_structpm} ruling out structured null vectors.
These propositions are proved in Section \ref{sec_struct}.
We remark that the proof of Theorem \ref{thm_gen} is not needed for the proof of Theorem \ref{thm_main}, so the reader who is only interested in the proof of the main theorem can begin at Section \ref{sec_reductions}.

\subsection{Warmup: Proof of Theorem \ref{thm_gen}} 
\label{sec_pm}

We restrict the sample space to the event $\good(d)$ defined in Theorem \ref{thm_gen}.
For convenience, we let $\good_i(d)$ denote the event that condition $i$ of $\good(d)$ holds for $\Sigma$, so that 
$\good(d)=\bigwedge_{i=0}^3\good_i(d).$
In this section we will only use the ``minimum degree" and ``no thin dense minors" properties enjoyed on $\good_0(d)\wedge\good_3(d)$. 
We denote the rows of $\Sigma$ by $r_i$ and the rows of $\Xi$ by $Y_i$, so that the $i$th row of $H$ is $R_i=r_i\schur Y_i$.
Our aim is to control the event
\begin{equation}	\label{mR1def}
\mR_1 =\set{\corank(H)\ge 1}.
\end{equation}
The following lemma reduces this task to bounding the event that a randomly sampled row lands in the span of the remaining rows.
We will extend this to larger corank with Lemma \ref{lem_sampling}.
Recall that a vector $x\in \R^n$ is $k$-sparse if $|\spt(x)|\le k$. 

\begin{lemma}
\label{lem_sampling1}
Let $H$ be a random $n\times n$ matrix with rows $R_i$, $1\le i\le n$. 
For $\eta\in (0,1)$, let $\good_{L}(\eta)$ be the event that $H$ has no non-trivial $(1-\eta)n$-sparse left null vectors. 
For $i\in [n]$ denote $V_i:= \Span(R_j: j\ne i)$, and define the events
\begin{equation}
\mS_i:=\big\{ R_i\in V_i\}.
\end{equation}
Draw $I$ uniformly from $[n]$, independently of $H$.
Then with $\mR_1$ as in \eqref{mR1def} we have
\begin{equation}
\pr\big(\, \mR_1\wedge \good_{L}(\eta) \, \big) \le \frac{1}{1-\eta} \pr\big(\, \mS_I\wedge\good_L(\eta)\,\big).
\end{equation}
\end{lemma}

\begin{proof}
On $\mR_1\wedge \good_L(\eta)$, $H$ has a left null vector with support of size at least $(1-\eta)n$.
It follows that on this event, $\mS_i$ holds for at least $(1-\eta)n$ values of $i\in [n]$. 
By double counting,
\begin{equation}
(1-\eta)n \pr\big(\mR_1\wedge \good_L(\eta)\big) \le \sum_{i=1}^n \pr\big( \mS_i\wedge \good_L(\eta)\big)
\end{equation}
and the result follows by rearranging. 
\end{proof}

From Proposition \ref{prop_structpm} we have that $\good_\pm^{\sparse}(\eta)$ holds 
with probability $1-O(n^{-100})$
for $H=\Sigma\schur\Xi$ for any 
$\eta\in [C_1'd^{-1/4},1]$, where $C_1'>0$ is a sufficiently large absolute constant.
Since $\good^{\sparse}_\pm(\eta)$ is simply the event that $\good_L(\eta)$ holds for $H$ and $H^\tran$, by the above lemma it suffices to show
\begin{equation}
\pr\big( \mS_I\wedge \good^{\sparse}_\pm(\eta)\big) \ll \kappa_3\eta + d^{-1/2}
\end{equation}
for arbitrary $\eta\in [C_1'd^{-1/4}, 0.1]$ (say), where $I$ is drawn uniformly of $[n]$, independently of $\Sigma$ and $\Xi$.
From now on we restrict the sample space to $\good^{\sparse}_\pm(\eta)$ for $\eta$ in this range, in order to lighten the notation.

Draw $u$ uniformly from the unit sphere in $V_I^\perp$, 
in a way such that $u$, $r_I$ and $Y_I$ are jointly independent conditional on $I$ and the remaining rows of $\Sigma$ and $\Xi$.
Now it would be enough to show
\begin{equation}
\pr\Big\{(r_I\schur Y_I)\cdot u = 0\Big\} \ll \kappa_3\eta + d^{-1/2}.
\end{equation}
From Theorem \ref{thm_erdos} we have
\begin{equation}	\label{allgood}
\pros[Y_I]{(r_I\schur Y_I)\cdot u = 0 \,\big|\, \Sigma, I}  \ll \left| \spt(u)\cap \spt(r_I) \right|^{-1/2}
\end{equation}
so we need to argue that $\spt(u)$ and $\spt(r_I)=\mN_\Sigma(I)$ have large overlap. 

By our restriction to $\good_0(d)$ we have $|\mN_\Sigma(i)|\ge d$ for all $i\in [n]$. 
We identify the set of undesirable realizations of $u$ as
\begin{equation}
\mH_{\Sigma,I}:= \set{x\in \R^n: \; \big|\spt(x)\cap \mN_\Sigma(I)\big|\le \frac{d}{2}}
\end{equation}
and define the bad event
\begin{equation}
\bad= \big\{ \pros[u]{u\in \mH_{\Sigma,I}}>0\big\}.
\end{equation}
We note that $\bad$ is decided by the randomness of $\Sigma$, $I$, and the rows $(Y_i)_{i\ne I}$ of $\Xi$. 

First we bound $\pro{ \mS_I \wedge\bad }$ using the randomness of $I$ and our restriction to $\good_3(d)\wedge\good^{\sparse}_\pm(\eta)$. 
The crucial observation is that $\mH_{\Sigma,I}$ is a finite union of subspaces, each of co-dimension at least $d/2$:
$$\mH_{\Sigma,I} = \bigcup_{B\subset \mN_\Sigma(I):\, |B|\le \frac{d}{2}} \R^{([n]\setminus \mN_\Sigma(I))\cup B}.$$
Since we picked $u$ according to the surface measure on the unit sphere of $V_{I}^\perp$, it follows that on $\bad$ we actually have
$V_I^\perp \subset \mH_{\Sigma,I}.$
On $\mR_1$ we may pick a nontrivial vector $x\in \ker(H)$ (note that the kernel is nontrivial on this event).
Crucially, we may do this with $x$ independent of $I$.
We have
$$x\in \ker(H) \subset V_I^\perp$$
and so on $\bad$ we have 
\begin{equation}
x\in \mH_{\Sigma,I}.
\end{equation}
Summarizing our progress so far,
\begin{equation}		\label{progress}
\pro{ \mS_I \wedge\bad } \le
\pro{ \mR_1 \wedge\big\{x\in \mH_{\Sigma,I}\big\} }
\end{equation}
where we used that $\mS_I\subset \mR_1$.
Letting
\begin{equation}
S_\Sigma(x) = \set{i\in [n]: \big|\mN_\Sigma(i)\cap \spt(x)\big| \le \frac{d}{2}}
\end{equation}
we have
\begin{align*}
|S_\Sigma(x)|\frac{d}{2} &< \sum_{i\in S_{\Sigma}(x)} \big|\mN_\Sigma(i)\cap x^{-1}(0)\big| \\
&= e_\Sigma\big(S_\Sigma(x),x^{-1}(0)\big) \\
& \le \kappa_3d|x^{-1}(0)|
\end{align*}
where in the last line we applied our restriction to $\good_3(d)$.
By our restriction to $\good_\pm^{\sparse}(\eta)$ we conclude
\begin{equation}
|S_\Sigma(x)| \le 2\kappa_3|x^{-1}(0)| \le 2\kappa_3\eta n.
\end{equation}
It follows that conditional on $\Sigma$ and $\Xi$ such that $\mR_1$ holds, 
$$\pros[I]{x\in \mH_{\Sigma,I}} = \pros[I]{I\in S_\Sigma(x)} \le 2\kappa_3\eta$$
and so we conclude from \eqref{progress} that
\begin{equation} 	\label{beta}
\pro{ \mS_I \wedge\bad } \le 2\kappa_3\eta.
\end{equation}

It remains to bound $\pro{ \mS_I\wedge\bad^c }.$
Condition on $I$, $\Sigma$ and $(Y_i)_{i\ne I}$ such that $\bad$ does not hold.
Off a null event we may assume that $u\notin \mH_{\Sigma,I}$. 
Now since 
$$\big|\spt(u)\cap \mN_\Sigma(I)\big| > \frac{d}{2}$$
applying \eqref{allgood} we have  
\begin{align*}
\pr\big( \mS_I\wedge\bad^c \big) &\le \pro{\big\{ (r_I\schur Y_I)\cdot u=0\big\} \wedge \big\{ u\notin\mH_{\Sigma,I} \big\} }\\
&= \e \pr_{Y_I}\Big( (r_I\schur Y_I)\cdot u=0 \,\Big|\, \Sigma, I, (Y_i)_{i\ne I}\Big)\un(u\notin \mH_{\Sigma,I})\\
&\ll d^{-1/2}
\end{align*}
which combines with \eqref{beta} to give the claim. \qed

\subsection{Preliminary reductions}
\label{sec_reductions}

Now we turn to the proof of Theorem \ref{thm_main}.
We may assume 
\begin{equation}	\label{dlub}
C_0\log^2n\le d\le n/2
\end{equation}
(for the upper bound see Remark \ref{rmk_assumed}).
This will allow us to restrict to the following ``good events":
\begin{itemize}
\item By Theorem \ref{thm_codeg}, the event $\good^{\ex}(\delta)$ holds with overwhelming probability for any fixed $\delta\in (0,1)$ independent of $n$
(here we only need $d=\omega(\log n)$).
\item  From Proposition \ref{prop_struct} 
(which we are assuming for now, deferring the prove to Section \ref{sec_struct}) 
\eqref{dlub} implies that $\good^{\sls}(\eta)$ holds 
with probability $1-O(n^{-100})$ for any 
$\eta\in [C_1d^{-c_0},1]$.
Moreover, we note that  $\good^{\sls}(\eta)\subset\good^{\sparse}(\eta)$, the event that $M$ has no left or right $(1-\eta)n$-sparse null vectors. 
\end{itemize}
We leave the parameters $\delta,\eta\in (0,1)$ unspecified for now.

For $k\in [n]$ define the event
\begin{equation}		\label{Rkdef}
\mR_{k}:=\set{ \corank(M)\ge k}.
\end{equation}
Our aim is to bound $\pr(\mR_1)$. 
Unlike the proof for $H$ in the previous section, we will need to separately handle $\mR_2$ and $\mR_1\setminus \mR_2$ by different arguments.
The argument for $\mR_2$ will follow a similar approach to the proof of Theorem \ref{thm_gen}, after invoking the shuffling coupling to inject iid signs. 
Controlling $\mR_1\setminus \mR_2$ will require more care.
Theorem \ref{thm_main} follows from the next proposition and Proposition \ref{prop_struct}.

\begin{proposition}
\label{prop_red}
For all $\eta\in (0,1]$ 
we have
\begin{equation}	\label{corank2}
\pro{\mR_2 \wedge \good^{\sls}(\eta)} \ll \eta+ d^{-1/2}
\end{equation}
and
\begin{equation}		\label{corank1}
\pro{ \mR_1\wedge \mR_2^c \wedge \good^{\sls}(\eta)} \ll \eta+ d^{-1/2}.
\end{equation}
\end{proposition}

We will use the following extension of Lemma \ref{lem_sampling1} for controlling the event $\mR_k$ when there are no sparse null vectors. 
We only need this for $k=2$, but the result for larger values of $k$ comes with little additional effort.

\begin{lemma}[Control by random sampling]				
\label{lem_sampling}
Assume $M$ is a random $n\times n$ matrix with rows $R_i$, $1\le i\le n$. 
Let $\mR_k$ be as in \eqref{Rkdef} and for $\eta\in (0,1)$ let $\good_L(\eta)$ be the event that $M$ has no non-trivial $(1-\eta)n$-sparse left null vectors. 
For an arbitrary $k$-tuple of row indices $(i_1,\dots, i_k)\in [n]^k$, denote the subspaces 
\begin{equation}	\label{Vcomp}
V_{(i_1,\dots,i_k)} = \Span\Big(R_i:i\notin\set{i_l}_{l=1}^k\Big)
\end{equation}
and the events
$$\mS_{(i_1,\dots,i_k)}=\set{R_{i_1},\dots,R_{i_k}\in V_{(i_1,\dots,i_k)}}.$$
For $k\in [n/2]$, let $\mI=(I_1,\dots,I_k)$ be a vector of indices sampled uniformly without replacement from $[n]$, independently of $M$.
Then if $\eta\in \big(0, \frac1{2}\big)$, we have
$$\pr\big(\mR_k\wedge \good_L(\eta)\big) \le (1-2\eta )^{-k}\pr\Big(\mS_\mI\wedge \good_L(\eta)\Big).$$
\end{lemma}

\begin{proof}
Since
$$\pr\Big(\mS_\mI\wedge \good_L(\eta)\Big) = \pr\Big(\mS_\mI \,\Big|\, \mR_k\wedge\good_L(\eta)\Big) \pr\Big(\mR_k\wedge \good_L(\eta)\Big)$$
it suffices to show that for fixed $M$ such that $\mR_k\wedge \good_L(\eta)$ holds, we have
\begin{equation}
\pr_{\mI}\big(\mS_\mI\big)\ge (1-2\eta )^{-k}.
\end{equation}

Condition on such $M$. 
We may pick $k$ linearly independent left null vectors $y_1,\dots, y_k$, so that for each $j\in [k]$, 
$$\sum_{i\in [n]}y_j(i)R_i=0.$$
Next we apply row reduction to the $k\times n$ matrix with rows $y_j$.
For $\mS_\mI$ to hold, it suffices that there exist $Z=(z_1,\dots, z_k)\in \R^{n\times k}$ with $\Span(z_1,\dots, z_k)=\Span(y_1,\dots, y_k)$ such that the $k\times k$ matrix 
\begin{equation}		\label{zmat}
Z_{\mI\times [k]}= (z_j(I_l))_{1\le j,l,\le k}
\end{equation}
is upper triangular with nonzero diagonal entries. 
Indeed, this implies that $R_{I_1},\dots,R_{I_k}$ can each be expressed as linear combinations of the rows $\set{R_i: i\notin \set{I_1,\dots,I_k}}$. 

Set $z_1=y_1$, and let $\bad_1=\set{z_1(I_1)=0}$. 
By our restriction to $\good_L(\eta)$ we have
$$\pros[I_{1}]{\bad_{1}} \le \frac{|z_{1}^{-1}(0)|}{n}\le \eta.$$
For $j\in [k-1]$, having defined linearly independent vectors $z_1,\dots, z_j$ and events $\bad_1, \dots, \bad_j$, on $\bigwedge_{1\le l\le j}\bad_l^c$ we can find 
$z_{j+1}\in \Span(y_{j+1},z_j,\dots, z_1)$
such that $z_{j+1}(I_l)=0$ for all $1\le l\le j$ (by linear independence). 
Let  $\bad_{j+1}=\set{z_{j+1}(I_{j+1})=0}$. 
Since $z_{j+1}\in \ker(M^\tran)$, by our restriction to $\good_L(\eta)$ we have
$$\pr_{I_{j+1}}\bigg(\bad_{j+1}\Big|\bigwedge_{1\le l\le j}\bad_l^c\bigg) \le \frac{|z_{j+1}^{-1}(0)|}{n-j}\le 2\eta$$
(using the upper bound on $k$). 
Applying the above bound iteratively with Bayes' rule we conclude that $\bigwedge_{1\le l\le k}\bad_l^c$ holds with probability at least $(1-2\eta)^k$ in the randomness of $\mI$, and on this event the matrix \eqref{zmat} has the desired properties. 
\end{proof}

\subsection{Injecting a random walk}
\label{sec_inject}

We now turn to Proposition \ref{prop_red}.
Without the randomness of the independent signs enjoyed by $H$, we must use the shuffling coupling of Lemma \ref{lem_shuffling} to express $\mS_\mI$ as the event that a random walk lands at a particular point. 
We define a coupled pair of rrd matrices $(M,\tM)$ as in that lemma, but with the pair of rows selected randomly. 
That is, we draw:
\begin{enumerate}
\item an rrd matrix $M$, 
\item $I_1,I_2\in [n]$ sampled uniformly without replacement from $[n]$, independently of $M$,
\item a uniform random bijection
$$\pi:\Ex_M(I_1,I_2)\rightarrow \Ex_M(I_2,I_1),$$
\item a sequence $\xi:[n]\rightarrow \set{\pm1}$ of iid uniform signs independent of all other variables. 
\end{enumerate}
We form $\tM$ by performing a shuffling on $M$ at the rows $(I_1,I_2)$ with respect to $\pi$, $\xi$. 
By Lemma \ref{lem_shuffling} and conditioning on $I_1,I_2$ we have that $M\eqd\tM$. 

Now we wish to control the events $\mS_{I_1}=\set{R_{I_1}\in V_{I_1}}$
and 
$\mS_{(I_1,I_2)}=\set{R_{I_1},R_{I_2}\in V_{(I_1,I_2)}}.$
Note that on $\mS_{(I_1,I_2)}$, $R_{I_1}$ and $R_{I_2}$ are orthogonal to any vector in the orthocomplement of $V_{(I_1,I_2)}$. 
We are hence interested in the dot products $R_{I_1}\cdot u$, $R_{I_2}\cdot u$ for $u$ taken from the unit sphere (say) of $V_{(I_1,I_2)}^\perp$. 
Let us examine the joint distribution of these dot products when we replace $M$ by $\tM$. 
Letting $\tR_i$ denote the $i$th row of $\tM$, we can express
\begin{align}
\left( \begin{array}c \tR_{I_1}\cdot u\\ \tR_{I_2}\cdot u  \end{array} \right)
&= \left( \begin{array}c
\sum_{j\in \mN_{\tM}(I_1)} u(j) \\
\sum_{j\in \mN_{\tM}(I_2)} u(j)
\end{array} \right) \notag\\
&= 
\left( \begin{array}c
\sum_{j\in \Co_M(I_1,I_2)}u(j)+\sum_{j\in \Ex_{\tM}(I_1,I_2)}u(j) \\
 \sum_{j\in \Co_M(I_1,I_2)}u(j)+\sum_{j\in \Ex_{\tM}(I_2,I_1)}u(j)
\end{array} \right) \notag\\
&= \sum_{j\in \Co_M(I_1,I_2)}u(j)  {1\choose 1} + \sum_{j\in \Ex_M(I_1,I_2)} 
\left( \begin{array}c
 u\big(j\un(\xi(j)=+1)+\pi(j)\un(\xi(j)=-1)\big) \\
  u\big(\pi(j)\un(\xi(j)=+1) + j\un(\xi(j)=-1)\big) 
\end{array} \right) \notag\\
&= \sum_{j\in \Co_M(I_1,I_2)}u(j)  {1\choose 1}
+ \sum_{j\in \Ex_M(I_1,I_2)}\frac{u(j)+u(\pi(j))}{2}{1\choose 1} + \xi(j)\frac{u(j)-u(\pi(j))}{2}{1\choose -1} \notag\\
&= \bigg(\frac{1}{2}(R_{I_1}+R_{I_2})\cdot u \bigg) {1\choose 1} + \Bigg(\sum_{j\in \Ex_M(I_1,I_2)}\xi(j) \partial _j^\pi(u)\Bigg){1\choose -1}	\label{rwalk1} \\
&=: A(u){1\choose 1}+W(u){1\choose -1}		\label{rwalk2}
\end{align}
where in the penultimate line we have defined 
\begin{equation}
\partial_j^\pi(u)=\frac{u(j)-u(\pi(j))}{2}.
\end{equation}

Note that the term $A(u)$ is fixed by conditioning on $M, I_1,I_2$.  
Furthermore, the sequence $(\partial_j^\pi(u))_{j\in \Ex_M(I_1,I_2)}$ is fixed by additionally conditioning on $\pi$. 
Hence, conditional on $M,I_1,I_2,\pi$, in the randomness of the $\xi(j)$ this pair of dot products is a random walk in the $(1, -1)$ direction with steps $\partial_j^{\pi}(u)$. 

The following lemma isolates the role of the randomness of the signs $\xi(j)$ and reduces the problem to the study of structural properties of the normal vector $u$. 
While it is stated for an arbitrary fixed pair of row indices $(i_1,i_2)$, it can be applied to the random pair $(I_1,I_2)$ after conditioning.

\begin{lemma}[The role of the signs $\xi(j)$]		
\label{lem_rwalk}
For $u\in \R^n$, $i_1,i_2\in [n]$ distinct, and a bijection
$\pi:\Ex_M(i_1,i_2)\rightarrow \Ex_M(i_2,i_1)$,
define
\begin{equation}	\label{stepsdef}
\Steps(u)=\Steps_{M,\pi}^{(i_1,i_2)}(u):=\set{j\in \Ex_M(i_1,i_2): u(j)\ne u(\pi(j))}.
\end{equation}
Then with $(M,\tM)$ coupled as in Lemma \ref{lem_shuffling} and $u$ deterministic or random depending only on $(R_i:i\notin \set{i_1,i_2})$ we have
\begin{equation}
\pr_{\xi}\big(\tR_{i_1}\cdot u=0\big) \ll |\Steps(u)|^{-1/2}.
\end{equation}
\end{lemma}

\begin{proof}
From the representation \eqref{rwalk1} we have
$$\tR_{i_1}\cdot u = A(u) + \sum_{j\in \Ex_M(i_1,i_2)}\xi(j) \partial _j^\pi(u)$$
and the claim follows by conditioning on $M, \pi$ and applying Theorem \ref{thm_erdos}.
\end{proof}

\subsection{Ruling out corank $\ge 2$}
\label{sec_corank2}

In this section we establish the bound \eqref{corank2} from Proposition \ref{prop_red}.
By increasing the hidden constant in \eqref{corank2} we may assume $\eta$ is at most a sufficiently small absolute constant. 

From Lemma \ref{lem_sampling}, for $\eta$ sufficiently small it suffices to bound $\pr(\mS_{(I_1,I_2)})$. 
Conditional on $M,I_1,I_2$, let $u$ be drawn from the uniform surface measure of the unit sphere in $V_{(I_1,I_2)}^\perp$, independently of $\pi, R_{I_1}, R_{I_2}$. 
We have
\begin{align*}
\pr(\mS_{(I_1,I_2)})& \le \pro{\tR_{I_1}\in V_{(I_1,I_2)}} \\
&\le \pro{\tR_{I_1}\cdot u=0}.
\end{align*}
We want to bound this using Lemma \ref{lem_rwalk}, so we will need to argue that 
the set $\Steps_{M,\pi}^{(I_1,I_2)}(u)$ defined there is large.
For this task we use the randomness of $u,I_1,I_2,\pi$, and restrict $M$ to the good events $\good^{\ex}(\delta)$ and $\good^{\sls}(\eta)$. 

First we identify the set of undesirable realizations of $u$. 
Let
\begin{equation}
\mH'_{M,I_1,I_2}=\set{x\in \R^n:\; \exists \lambda\in \R \mbox{ with } \min_{l=1,2}|\mN_M(I_l)\cap x^{-1}(\lambda)|> d/100}.
\end{equation}
That is, $\mH'_{M,I_1,I_2}$ is the set of vectors with a level set intersecting at least $1\%$ of the support of both $R_{I_1}$ and $R_{I_2}$. 
Note that $\mH'_{M,I_1,I_2}$ is a finite union of proper subspaces of $\R^n$. 
Indeed, we may express
$$\mH'_{M,I_1,I_2}=\bigcup_{(T_1,T_2)} \mH_{T_1\cup T_2}$$
where the union ranges over pairs of subsets $T_1\subset \mN_M(I_1), T_2\subset \mN_M(I_2)$ of size at least $d/100$, and $\mH_T$ denotes the subspace of vectors that are constant on $T$. 
Define the bad event
\begin{equation} \label{bmi}
\bad' = \set{\pr_u(u\in \mH'_{M,I_1,I_2})>0}.
\end{equation}

Since $\mH'_{M,I_1,I_2}$ is a finite union of proper subspaces, and $u$ is drawn from the uniform surface measure of the subspaces $V_{(I_1,I_2)}^\perp$, it follows that if $\bad'$ holds then we actually have the inclusion
$$V_{(I_1,I_2)}^\perp \subset \mH'_{M,I_1,I_2}.$$
Note also that $\ker(M)\subset V_{(I_1,I_2)}^\perp$.
On $\mR_2$, we may fix an arbitrary nontrivial element $x\in \ker(M)$ (note that the kernel is nonempty on this event), independent of $I_1,I_2$. 
Now we have
\begin{align*}
\pr(\mR_2\wedge\bad') &\le \pr(\mR_2\wedge\set{x\in \mH'_{M,I_1,I_2}}).
\end{align*}
We will bound the latter quantity using the randomness of $I_1,I_2$. 
For $\lambda\in \R$, let
$$S_M(x,\lambda)=\set{i\in [n]: |\mN_M(i)\cap x^{-1}(\lambda)|\ge d/100}.$$
We can control the size of these sets using only a crude bound on edge counts:
$$|S_M(x,\lambda)|\frac{d}{100}< e_M(S_M(x,\lambda),x^{-1}(\lambda)) \le d|x^{-1}(\lambda)|$$
whence
$$|S_M(x,\lambda)|\le 100 |x^{-1}(\lambda)|.$$
Now for $M$ such that $\good^{\sls}(\eta)$ holds we have $|x^{-1}(\lambda)|\le \eta n$ for all $\lambda\in \R$.
Conditional on $M$ such that $\mR_2$ and $\good^{\sls}(\eta)$ hold (which fixes $x\ne 0$), we can bound
\begin{align*}
\pr_{I_1,I_2}\big(x\in \mH'_{M,I_1,I_2}\big) &= \pr_{I_1,I_2}\big(\exists \lambda\in \R: I_1,I_2\in S_M(x,\lambda)\big) \\
&\le \sum_{\lambda: x^{-1}(\lambda)\ne \phi}\pr_{I_1,I_2}(I_1,I_2\in S_M(x,\lambda)) \\
&\ll \sum_{\lambda:x^{-1}(\lambda)\ne \phi} \left(\frac{|x^{-1}(\lambda)|}{n}\right)^2\\
&\le \eta\sum_{\lambda:x^{-1}(\lambda)\ne \phi} \frac{|x^{-1}(\lambda)|}{n} \\
&\le \eta.
\end{align*}
Undoing the conditioning on $M$, we have shown that
\begin{equation}		\label{bmibound}
\pr(\bad'\wedge \mR_2\wedge \good^{\sls}(\eta)) \ll \eta.
\end{equation}

It remains to bound
$\pro{\big\{R_{I_1}\cdot u=0\big\} \setminus \bad'}.$
Condition on $M,{I_1,I_2}$ such that $\bad'$ does not hold. 
Off a null event we may assume that $u\notin \mH'_{M,{I_1,I_2}}$. 
That is, for every $\lambda\in \R$, we may assume 
\begin{equation}	\label{alternative}
|\mN_M(I_1)\cap u^{-1}(\lambda)|\le d/100\quad\mbox{or} \quad |\mN_M(I_2)\cap u^{-1}(\lambda)|\le d/100.
\end{equation}
We will now get a lower bound on $|\Steps(u)|$ (as defined in \eqref{stepsdef}). 
It will be more convenient to work with the complementary set
\begin{align}
\Flats(u) &:= \Ex_M(I_1,I_2)\setminus \Steps(u) \label{flatsdef}\\
&= \set{j\in \Ex_M(I_1,I_2): u(j)=u(\pi(j))}.	\notag
\end{align}
We have
\begin{align*}
\e_\pi |\Flats(u)| &= \sum_{j\in \Ex_M(I_1,I_2)} \pr_\pi(\pi(j)\in u^{-1}(u(j))) \\
&= \sum_{j\in \Ex_M(I_1,I_2)} \frac{|u^{-1}(u(j))\cap \Ex_M(I_2,I_1)|}{|\Ex_M(I_2,I_1)|} \\
&= \frac{1}{|\Ex_M(I_1,I_2)|}\sum_{\lambda:u^{-1}(\lambda)\ne \phi} |u^{-1}(\lambda)\cap \Ex_M(I_1,I_2)||u^{-1}(\lambda)\cap \Ex_M(I_2,I_1)|
\end{align*}
where the last line follows from double counting. 
Now we apply \eqref{alternative} to get
\begin{align}
\e_\pi|\Flats(u)| & \le \frac{d}{100}\frac{1}{|\Ex_M(I_1,I_2)|}\sum_{\lambda:u^{-1}(\lambda)\ne \phi} \max\big(|u^{-1}(\lambda)\cap\Ex_M(I_1,I_2)|,|u^{-1}(\lambda)\cap \Ex_M(I_2,I_1)|\big) \notag \\
&\le \frac{d}{100}\frac{1}{|\Ex_M(I_1,I_2)|}\sum_{\lambda:u^{-1}(\lambda)\ne \phi}\big|u^{-1}(\lambda)\cap[\Ex_M(I_1,I_2)\cup \Ex_M(I_2,I_1)]\big|\notag\\
&= \frac{d}{50}.		\label{exflat}
\end{align}
We want to show that $|\Flats(u)|$ is concentrated around its expectation (we only need control on the upper tail). 
In the notation of Lemma \ref{lem_com} we have
$
|\Flats(u)| = e_\pi(A,B)
$
with $A=\Ex_M(I_1,I_2)$ and $B=u^{-1}(u(j_1))$ (which are fixed by conditioning on $M,I_1,I_2$).
Applying Lemma \ref{lem_com} and \eqref{exflat} we conclude that for any $\eps>0$, 
\begin{align}
\pros[\pi]{|\Flats(u)|\ge (1+\eps)\frac{d}{50}} &\le \pros[\pi]{|\Flats(u)|\ge \e_\pi|\Flats(u)|+ \eps\frac{d}{50}} \notag\\
&\ll \expo{ -\frac{c\eps^2d^2/50^2}{\e_\pi|\Flats(u)|+\eps \frac{d}{50}}} \notag\\
&\le \expo{-\frac{c\eps^2}{1+\eps}d}.		\label{flatstail}
\end{align}
On the other hand, on $\good^{\ex}(\delta)$ we have (applying our assumption $d\le n/2$)
$$|\Ex_M(I_1,I_2)|\ge (1-\delta)d\Big(1-\frac{d}{n}\Big)\ge \frac{1-\delta}{2}d$$
so that fixing $\delta\le 1/2$ and $\eps\le 4$ (say), we conclude that on $\bad'^c\wedge\good^{\ex}(\delta)$, except with probability at most $O(\exp(-cd))$ we have
\begin{equation}		\label{stepslb}
|\Steps(u)|=|\Ex_M(I_1,I_2)|-|\Flats(u)|\ge \frac{d}{10}.
\end{equation}

Summarizing our work so far, 
\begin{align*}
\pr\Big[\,\mS_{(I_1,I_2)}\wedge \good^{\sls}(\eta)\wedge\good^{\ex}(\delta)\,\Big] &\le 
\pr\Big[\,\mR_2\wedge\bad'\wedge \good^{\sls}(\eta)\,\Big] + \pr\Big[\,\mS_{(I_1,I_2)}\wedge \bad'^c\wedge \good^{\sls}(\eta)\wedge\good^{\ex}(\delta)\,\Big] \\
&\le \pr\Big[\,\mR_2\wedge\bad'\wedge \good^{\sls}(\eta)\,\Big] + 
\pr\bigg[\,\bad'^c\wedge\set{|\Steps(u)|<\frac{d}{10}}\wedge\good^{\ex}(\delta)\,\bigg] \\
&\quad \quad\quad+\pr\bigg[\,\Big\{\tR_{I_1}\cdot u=0\Big\}\wedge \set{|\Steps(u)|\ge \frac{d}{10}}\,\bigg]\\
&\ll \eta + e^{-cd} + d^{-1/2}
\end{align*}
where in the last line we substituted our bounds \eqref{stepslb} and \eqref{bmibound} and applied Lemma \ref{lem_rwalk}. 
Combined with our estimates for $\good^{\sls}(\eta)$ and $\good^{\ex}(\delta)$, together with Lemma \ref{lem_sampling} we have
\begin{equation}
\pr(\mR_2) \ll \eta+d^{-1/2}
\end{equation}
as desired.

\subsection{Ruling out corank $1$}
\label{sec_corank1}

In this section we establish the bound \eqref{corank1}, which completes the proof of Proposition \ref{prop_red} and hence of Theorem \ref{thm_main}.
By increasing the hidden constant in \eqref{corank1} we may assume $\eta$ is at most a sufficiently small absolute constant. 

By Lemma \ref{lem_sampling} it suffices to bound 
$$\pr(\mS_{I_1}\setminus \mR_2) = \pr\big(R_{I_1}\in V_{I_1}, \;  \corank(M)= 1\big)$$
(taking $\eta$ smaller if necessary).
 We cannot simply condition on all rows but $R_{I_1}$ and pick a normal vector $u\in V_{I_1}^\perp$, since this conditioning fixes $R_{I_1}$ as well by $d$-regularity. 
Instead, we will leave $R_{I_2}$ random and express the event $\mS_{I_1}$ in terms of a certain $2\times2$ determinant involving the rows $R_{I_1},R_{I_2}$.
We now have the advantage that on the bad event, we can condition on a unique (up to dilation) null vector $x$ of $M$, which is independent of ${I_1,I_2}$. 

We turn to the details. 
Conditional on $M,I_1,I_2$, we pick a pair of orthonormal vectors $u_1\perp u_2\in V_{(I_1,I_2)}^\perp$ uniformly at random, and independently of $(R_{I_1},R_{I_2})$. 
(On $\mS_{I_1}\wedge \mR_2^c$ we have $\dim(V_{(I_1,I_2)})=n-2$, so $u_1,u_2$ are an orthonormal basis for $V_{(I_1,I_2)}^\perp$ on this event.)
In terms of $u_1,u_2$ we may construct a vector which is also orthogonal to $R_{I_2}$ as follows:
$$z_1:= (R_{I_2}\cdot u_2)u_1-(R_{I_2}\cdot u_1)u_2 \; \in V_{I_1}^\perp.$$
Since $z_1$ lies in the orthocomplement of $V_{I_1}$, on $\mS_{I_1}$ we have
\begin{align}
0&= R_{I_1}\cdot z_1 \notag\\
&= (R_{I_1}\cdot u_1)(R_{I_2}\cdot u_2)-(R_{I_2}\cdot u_1)(R_{I_1}\cdot u_2)		\label{det}\\
&=:D_M(I_1,I_2).	\notag
\end{align}
Hence,
\begin{equation}	\label{detbound}
\pr\big(\mS_{I_1}\setminus \mR_2\big) \le \pr\big(D_M(I_1,I_2)=0,\; \corank(M)=1\big).
\end{equation}

Substituting $\tM$ for $M$, may express the $2\times 2$ determinant using \eqref{rwalk2}:
\begin{align}
D_{\tM}(I_1,I_2) &= \big[A(u_1)+W(u_1)\big]\big[A(u_2)-W(u_2)\big] - \big[A(u_2)+W(u_2)\big]\big[A(u_1)-W(u_1)\big]\notag\\
&= 2A(u_2)W(u_1)-2A(u_1)W(u_2)\notag\\
&= \sum_{j\in \Ex_M(I_1,I_2)}\xi(j) \big[ 2A(u_2)\partial_j^\pi(u_1)- 2A(u_2)\partial_j^{\pi}(u_2)\big]\notag\\
&= \sum_{j\in \Ex_M(I_1,I_2)}\xi(j)\partial_j^\pi(v )	\label{detwalk1}\\
&= W(v )		\label{detwalk2}
\end{align}
where we have defined
\begin{align}
 v  &:= 2A(u_2)u_1- 2A(u_2)u_2 \notag\\
&= \big[(R_{I_1}+R_{I_2})\cdot u_2\big] u_1- \big[(R_{I_1}+R_{I_2})\cdot u_1\big]u_2	\label{vdef}\\
&\in V_{(I_1,I_2)}^\perp.	\notag
\end{align} 

We would like to replace $M$ with $\tM$ and bound \eqref{detbound} using the random walk representation \eqref{detwalk2} with Theorem \ref{thm_erdos}. 
First we must reduce to an event on which many of the steps $\partial_j^\pi(v )$ are nonzero. 
We will do this in two stages.
First we must remove a bad event on which $v =0$; in light of \eqref{vdef} this is the event
$$\bad_0 := \set{R_{I_1}+R_{I_2}\in V_{(I_1,I_2)}}.$$
Once we have done this, we will be able to argue that $v $ is unstructured in a manner similar to the way we controlled the event $\bad'$ in Section \ref{sec_corank2}.

We begin with $\bad_0$. 
Since we are free to restrict to $\mR_2^c\wedge\good_{\eta}^{\sls}$, let us condition on $M$ such that these events hold. 
On $\bad_0\wedge\mR_2^c$, $M$ has exactly one nontrivial left null vector (up to dilation) which we denote by $y$; furthermore, on $\good^{\sls}(\eta)$ the level sets of $y$ are of size at most $\eta n$. 
Now $\bad_0$ is the event that $M$ has a left null vector $y'$ with $y'(I_1)=y'(I_2)\ne 0$, so we have
$$\bad_0\wedge \mR_2^c\subset \set{y(I_1)=y(I_2)}.$$
It follows that
\begin{align}
\pros[{I_1,I_2}]{\bad_0\wedge \mR_2^c\wedge\good^{\sls}(\eta)} &\le \pr_{{I_1,I_2}}\big(y(I_1)=y(I_2)\big) \notag\\
&\le \eta(1+o(1))		\label{badpmi}
\end{align}
which is small enough.

Similarly to what we did in Section \ref{sec_corank2}, for $\eps_1\in (0,1)$ we define
\begin{equation}	\label{HMI}
\mH'_{M,{I_1,I_2}}(\eps_1)= \set{x\in \R^n:\exists \lambda\in \R\mbox{ with } \min_{l=1,2}\big|x^{-1}(\lambda)\cap\mN_M(I_l)\big|>\eps_1 d}
\end{equation}
but we also set for $\eps_2\in (0,1)$
\begin{equation}		\label{HMI'}
\mH''_{M,{I_1,I_2}}(\eps_2)= \Big\{x\in \R^n:\exists \lambda\in \R\mbox{ with } \big|x^{-1}(\lambda)\cap\big[\Ex_M(I_1,I_2)\cup \Ex_M(I_2,I_1)\big]\big|>\eps_2 p(1-p)n\Big\}.
\end{equation}
For $\eps\in (0,1)$ we have the inclusion
\begin{align}		
\mH''_{M,{I_1,I_2}}(1+2\eps) &\subset \mH'_{M,{I_1,I_2}}(2\eps(1-\pe)) \notag\\
&\subset \mH'_{M,{I_1,I_2}}(\eps) \label{inclusions}
\end{align}
(by our assumption $\pe\le 1/2$). 
Since $\mH''_{M,{I_1,I_2}}$ is determined by $\Ex_M(I_1,I_2)\cup \Ex_M(I_2,I_1)=\Ex_{\tM}(I_1,I_2)\cup \Ex_{\tM}(I_2,I_1)$, we have 
\begin{equation}	\label{hinvar}
\mH''_{M,{I_1,I_2}}(\eps)=\mH_{\tM,{I_1,I_2}}''(\eps)
\end{equation}
for any $\eps\in (0,1)$, whereas this invariance does not hold for $\mH'_{M,{I_1,I_2}}(\eps)$. 

Let $\eps>0$ to be chosen later. 
On $\set{\corank(M)=1}$, let $x$ denote a fixed nontrivial null vector of $M$, so that $\ker(M)=\langle x\rangle.$
From \eqref{det}, on $\set{D_M(I_1,I_2)=0}$ we have $z_1,z_2\in \ker(M)$, where
$$z_2:= (R_{I_1}\cdot u_1)u_2-(R_{I_1}\cdot u_2)u_1 \; \in V_{I_2}^\perp$$
and (as before) 
$$z_1:= (R_{I_2}\cdot u_2)u_1-(R_{I_2}\cdot u_1)u_2 \; \in V_{I_1}^\perp.$$
It follows that on $\set{D_M(I_1,I_2)=0}$ we have $v=z_1-z_2\in \ker(M)$. On the intersection of this event with $\set{\corank(M)=1}$ and the event $\bad_0^c$ on which $v$ is non-zero, we have $0\ne v\in \langle x\rangle$.
Hence,
\begin{align}
\bad_0^c\wedge\big\{D_M(I_1,I_2)=0,&\; \corank(M)=1, v \in \mH'_{M,{I_1,I_2}}(\eps)\big\}\notag\\
&\subset \set{\corank(M)=1,\;x\in \mH'_{M,{I_1,I_2}}(\eps)}.
\end{align}
We may now argue exactly as in Section \ref{sec_corank2} to conclude
\begin{equation}		\label{xinH}
\pros[{I_1,I_2}]{\good_{\eta}^{\sls}\wedge\set{\corank(M)=1,\;x\in \mH'_{M,{I_1,I_2}}(\eps)}}\ll_\eps \eta.
\end{equation}

It only remains to bound
\begin{equation}	\label{remains}
\pr\big(D_M(I_1,I_2)=0, \; v \notin \mH'_{M,{I_1,I_2}}(\eps)\big).
\end{equation}
From \eqref{inclusions} this is bounded by
$$\pr\big(D_M(I_1,I_2)=0,\; v \notin \mH''_{M,{I_1,I_2}}(1+2\eps)\big).$$
Now we replace $M$ with $\tM$. 
We make the crucial observation that the second event is unchanged by this substitution. 
Indeed, $\mH''_{M,{I_1,I_2}}(1+2\eps)$ is unchanged as noted in \eqref{hinvar}.
Similarly, $v $ is the same or $M$ and $\tM$ since
$$v =\big[(R_{I_1}+R_{I_2})\cdot u_2\big] u_1 - \big[(R_{I_1}+R_{I_2})\cdot u_1\big] u_2$$
and $R_{I_1}+R_{I_2}=\tR_{I_1}+\tR_{I_2}$ as the shuffling preserves the sets $\mN_M(I_1)\cap \mN_M(I_2)$ and $\mN_M(I_1)\cup \mN_M(I_2)$. 
Hence, \eqref{remains} is bounded by
\begin{equation}
\pr\Big(D_{\tM}(I_1,I_2)=0,\; v  \notin \mH''_{M,{I_1,I_2}}(1+2\eps)\Big) = 
\pr\Big(W(v )=0,\;   v  \notin \mH''_{M,{I_1,I_2}}(1+2\eps) \Big).
\end{equation}
In the final step of the argument, we must show that 
the set $\Steps(v)$ (as defined in \eqref{stepsdef})
is usually large off the event 
$$\bad'':=\set{v \in \mH''_{M,{I_1,I_2}}(1+2\eps)}$$
with high probability in the randomness of $\pi$ (and taking $\eps$ sufficiently small). 
Conditioning on $M$ and ${I_1,I_2}$ such that $\bad''$ does not hold, 
with $\Flats(v )$ as in \eqref{flatsdef} we have
\begin{align*}
\e_\pi \big|\Flats(v )\big| &= \sum_{j\in \Ex_M(I_1,I_2)}\pr_\pi\big(\pi(j)\in v ^{-1}(v (j))\big) \\
&= \sum_{j\in \Ex_M(I_1,I_2)}\frac{\big|v ^{-1}(v (j))\cap \Ex_M(I_2,I_1)\big|}{\big|\Ex_M(I_2,I_1)\big|} \\
&=\frac{1}{\ex_M(I_1,I_2)}\sum_{\lambda:v ^{-1}(\lambda)\ne \phi} \big|v ^{-1}(\lambda)\cap \Ex_M(I_1,I_2)\big|  \big|v ^{-1}(\lambda)\cap \Ex_M(I_2,I_1)\big| \\
&\le \frac{1}{\ex_M(I_1,I_2)}\sum_{\lambda:v ^{-1}(\lambda)\ne \phi} \frac{1}{4}\big|v ^{-1}(\lambda)\cap\big( \Ex_M(I_1,I_2)\cup \Ex_M(I_2,I_1)\big)\big|^2 \\
&\le \frac{(1+2\eps)\pe(1-\pe)n}{4\ex_M(I_1,I_2)}
\sum_{\lambda:v ^{-1}(\lambda)\ne \phi}\big|v ^{-1}(\lambda)\cap\big( \Ex_M(I_1,I_2)\cup \Ex_M(I_2,I_1)\big)\big| \\
&\le \frac{1+2\eps}{2}p(1-p)n.
\end{align*}
We can then argue exactly as in \eqref{flatstail} that
\begin{equation}
\pros[\pi]{ \big|\Flats(v )\big| \ge \Big(\frac{1}{2}+2\eps\Big)d\Big(1-\frac{d}{n}\Big)} \le \exp\bigg( -c\frac{\eps^2}{1+\eps}d\bigg)
\end{equation}
(substituting $p=d/n$).
On the other hand, on $\good^{\ex}(\delta)$ we have
$$\big|\Ex_M(I_1,I_2)\big|\ge (1-\delta)d\Big(1-\frac{d}{n}\Big)$$
so that if we take $\eps$ and $\delta$ sufficiently small, 
\begin{equation}
\big| \Steps(v )\big| = \big|\Ex_M(I_1,I_2)\big|-\big|\Flats(v )\big| \gg d
\end{equation}
(again using our assumption $d\le n/2$).
Applying Lemma \ref{lem_rwalk}, 
\begin{align*}
\pro{\set{W(v )=0} \wedge \good^{\ex}(\delta)\setminus \bad''} &\ll e^{-cd} + d^{-1/2}
\end{align*}
which combines with \eqref{badpmi} and \eqref{xinH} to give 
$$\pr(\mR_1\wedge \mR_2^c\wedge\good^{\sls}(\eta))\ll \eta + d^{-1/2}$$
as desired. \qed

\section{Structured null vectors}
\label{sec_struct}

Our aim in the section is to prove Propositions \ref{prop_structpm} and \ref{prop_struct}. 
The proof of the former outlines the proof of the latter and serves as a warmup. 
For the proof of Proposition \ref{prop_struct} we will use Lemma \ref{lem_shuffling} to inject random walks as in the previous section.
We will also make heavier use of the discrepancy properties from Section \ref{sec_discrepancy}.
We remark that the proof of Proposition \ref{prop_structpm} is not needed for the proof of Proposition \ref{prop_struct}, so the reader who is only interested in the proof of the main theorem can begin at Section \ref{sec_structprelim}.

\subsection{Warmup: no sparse null vectors for $H$ }
\label{sec_structpm}

In this section we prove Proposition \ref{prop_structpm}.
We restrict the sample space to the event $\good(d)$ as defined in Theorem \ref{thm_gen}.
Recall that $\good(d)$ is the event that for some constants $c_1,c_2,C_2>0$ and a parameter $\kappa_3\in [1,\infty)$ (possibly depending on $n$), the following conditions on $\Sigma$ hold:
\begin{enumerate}[start=0]
\item (Minimum degree) For all $i\in [n]$, $|\mN_\Sigma(i)|,|\mN_{\Sigma^\tran}(i)|\ge d$.
\item (Expansion of small sets) For all $\gamma\in (0,c_1]$, for all $S\subset[n]$ such that $|S|\le \frac{\log n}{2\gamma}\frac{n}{d}$, we have $|\mN_\Sigma(S)|, |\mN_{\Sigma^\tran}(S)|\ge \frac{\gamma}{\log n}d|S|$.
\item (No large sparse minors) For all $A,B\subset [n]$ such that $|A|,|B|\ge C_2\frac{n}{d}\log n$, we have $e_\Sigma(A,B)\ge c_2\frac{d}{n}|A||B|$.
\item (No thin dense minors) For any $S,B\subset [n]$, $e_\Sigma(S,B), e_\Sigma(B,S)\le \kappa_3d|S|$. 
\end{enumerate}
We assume $d\ge C_0'\log^2n$ for some $C_0'>0$ to be taken sufficiently large depending on $c_1,c_2,C_2$. 
As in Section \ref{sec_pm}, 
for $0\le i\le 3$
we let $\good_i(d)$ denote the event that condition $i$ above holds for $\Sigma$, so that $\good(d)=\bigwedge_{i=0}^3\good_i(d).$
We continue to denote the rows of $\Sigma$ by $r_i$ and the rows of $\Xi$ by $Y_i$, so that the $i$th row of $H$ is $R_i=r_i\schur Y_i$.

Since the event $\good(d)$ is the same if we replace $\Sigma$ with $\Sigma^\tran$, it suffices to consider only right null vectors. 
For $k\in [n]$, let
$$\event_k = \set{\exists x\in \R^n: 0<|\spt(x)|\le k, \, H x=0}.$$
Our goal is to show that $\event_{(1-\eta)n}^c$ holds 
with probability $1-O(n^{-100})$.
We have
\begin{equation}
\pro{ \event_{(1-\eta)n}} = \sum_{k=2}^{\lf (1-\eta)n\rf} \pro{\event_k\setminus \event_{k-1}}
\end{equation}
(noting that $\event_1$ is empty).
Fix $1\le k\le (1-\eta)n$.
We can follow the same lines establishing \eqref{line2} in the proof of Proposition \ref{prop_structber} to bound
\begin{equation}		\label{simplerpm}
\pro{\event_k\setminus \event_{k-1}}\le \sum_{S\in {[n]\choose k-1}}\; \sum_{T\in {[n]\choose k}} \pro{\event_{S,T}}
\end{equation}
where
$$\event_{S,T} := \Big\{ \exists x\in \R^n: \, H x = 0, \, \spt(x) = T, \, \set{R_i}_{i\in S} \mbox{ are linearly independent} \Big\}.$$
(We have \eqref{simplerpm} instead of \eqref{line2} since the rows and columns of $H$ are not exchangeable.)

Now we fix arbitrary $S,T\subset [n]$ of respective sizes $k-1$, $k$.
Fix also an arbitrary 
$v\in \R^n$ with support $T$. 
Since conditioning on $(R_i)_{i\in S}$ fixes $x$ on $\event_{S,T}$, it suffices to bound
$$\pro{ H_{S^c\times [n]}v=0 \, \Big| \, \Sigma , (Y_i)_{i\in S}}$$
uniformly in $v$.

Our approach is different depending on whether $k$ is small or large. 
In both cases, we use the fact that the rows $R_i$ decouple after conditioning on $\Sigma $:
\begin{align*}
\pro{ H_{S^c\times [n]}v=0 \,\Big|\, \Sigma } &= \prod_{i\in S^c} \pro{ R_i\cdot v = 0 \,\Big|\, \Sigma }\\
&=  \prod_{i\in S^c} \pro{ Y_i\cdot (r_i\schur v ) = 0 \,\Big|\, \Sigma }.		\label{pmdecouple}
\end{align*}
Now under this conditioning, the random variables 
$Y_i\cdot (r_i\schur v )$
are random walks (in the sense of Theorem \ref{thm_erdos}).
For small $k$, it will be enough to show that there are many $i\in S^c$ such that 
\begin{equation}	\label{spt1}
\big|\spt(R_i)\cap T\big|  = \big|\spt(r_i\schur v)\big| \ge 1.
\end{equation}
For such $i$, the random walk takes at least 1 nonzero step since 
$v(j)\ne 0$ 
for all $j\in T$, so we have 
$$\pr\big( Y_i\cdot (r_i\schur v ) = 0\big| \Sigma \big)\le 1/2$$
in this case.
To lower bound the number of rows satisfying \eqref{spt1} we will use our restriction to the ``expansion of small sets" event $\good_1(d)$.

For larger $k$ we will need to argue that the random walks 
$Y_i\cdot (r_i\schur v ) \big|\Sigma$
take more steps. 
For this we prove a consequence of our restriction to $\good_2(d)$ (Lemma \ref{lem_Aeps} below), which essentially guarantees that for most $i\in S^c$, the intersection of any sufficiently large set $B$ with the neighborhood $\mN_\Sigma (i)= \spt(r_i)$ has roughly its expected size, which by our restriction to $\good_0(d)$ is at least $p|B|$ (where we continue to denote $p:=d/n$).
Applying this with $B=T$ gives $|\spt(r_i)\cap T|\gg pk$ for most $i\in S^c$. 
We will build on this idea in the proof of Proposition \ref{prop_struct} for the unsigned rrd matrix $M$, where we will also need that a large set $B$ ``sees" roughly the expected portion of the sets $\Ex_M(i_1,i_2)$. 

We turn to the details.
Let $\gamma\in (0,c_1]$ to be chosen later.
First assume $k\le \frac1{2\gamma} \frac{n\log n}{d}$. 
Let 
$A_0 = \set{i\in S^c: r_i1_T\ne 0}.$
By our restriction to $\good_1(d)$ we have
\begin{align}
\left|A_0\right| &= |\mN_{\Sigma ^\tran}\big(T\big) \setminus S| \notag\\
&\ge \gamma\frac{dk}{\log n} -k.
\end{align}
Since 
$v(j)\ne 0$
for all $j\in T$, we have that for $i\in A_0$, 
$$\pr\big(Y_i\cdot (r_i\schur v ) = 0 \big| \Sigma , (Y_i)_{i\in S}\big) \le 1/2$$
whence
\begin{align*}
\pro{ H_{S^c\times [n]} v=0 \,\Big|\, \Sigma , (Y_i)_{i\in S}} &\le  \prod_{i\in A_0} \pro{Y_i\cdot (r_i\schur v ) = 0 \,\Big|\, \Sigma , (Y_i)_{i\in S} } \le \left( \frac12\right)^{\gamma \frac{dk}{\log n} - k}.
\end{align*}
Since this bound is uniform in $v, \set{r_i\schur Y_i}_{i\in S}$, we conclude from \eqref{simplerpm} that
\begin{align}
\pro{\event_k\setminus \event_{k-1}} &\le {n\choose k}{n\choose k-1} \left( \frac12\right)^{\gamma \frac{dk}{\log n} - k}
\notag \\
&\le \expo{ Ck\log n - c\gamma \frac{dk}{\log n}} \notag \\
&\le n^{-100k}\label{smallk}
\end{align}
where we have taken $C_0'>C'/\gamma$ for a sufficiently large constant $C'$ (we will later fix $\gamma$ depending on $C_2$), so that $d\ge \frac{C'}{\gamma}\log^2n$. 
Summing the bounds \eqref{smallk} gives
\begin{equation}	\label{klow}
\pro{\event_k} \ll n^{-100}
\end{equation}
for any $k\le \frac{1}{2\gamma}\frac{n\log n}{d}$.

Now assume $k> \frac{1}{2\gamma}\frac{n\log n}{d}$. 
For this case we apply the following consequence of our restriction to $\good_2(d)$.
(Recall that $\good_2(d)$ is the event that condition 2 from Theorem \ref{thm_gen} holds. Below we also make use of the constants $C_2,c_2$ defined there.)
\begin{lemma}		
\label{lem_Aeps}
For $A,B\subset [n]$ let
$$A' = \set{ i\in A: \; |B(i)| \ge c_2p|B|}$$
where we use the shorthand $B(i):=B\cap \mN_\Sigma (i)$, 
and denote $p:=d/n$.
On $\good_2(d)$, we have
$$ |A\setminus A'| \ll p^{-1}\log n$$
if $|B|\ge \frac{\log n}{2\gamma \pe} $ with $\gamma$ sufficiently small depending on $C_2$. 
\end{lemma} 

\begin{proof}
Define 
$$\mF= \set{(A,B):\; A, B\subset[n],\; \min(|A|,|B|) \ge C_2\frac{\log n}{p} }$$
so that on $\good_2(d)$ we have 
$e_\Sigma(A,B)\ge c_2p|A||B|$
for all $(A,B)\in \mF$.

Denote $S=A\setminus A'$.
We claim $(S,B)\notin \mF$. 
Indeed, if this were not the case we would have
\begin{align*}
c_2 p |S ||B| &\le e_\Sigma (S ,B) \\
& = \sum_{i\in S } |B(i)| \\
&< c_2 p|S ||B|
\end{align*}
a contradiction. 

Suppose $|B| \le |S |$. 
Since $(S ,B)\notin \mF$ we have
$$\frac{1}{2\gamma} \frac{n\log n}{d} \le |B| \le C_2 \frac{n\log n}{d}.$$
Taking $\gamma$ sufficiently small depending on $C_2$ we obtain a contradiction, and so $|S |\le |B|$, and by the definition of $\mF$ we must have $|S| \ll \frac{\log n}{p}.$
\end{proof}

Applying the lemma with $A=S^c$, $B=T$, 
we have that for all $i\in A'$, 
$|\spt(r_i)\cap T|\ge c_2pk$, and so by Theorem \ref{thm_erdos},
\begin{equation}
\pro{ Y_i\cdot (r_i\schur v ) = 0 \,\Big|\, \Sigma , (Y_i)_{i\in S}} \ll (pk )^{-1/2}.
\end{equation}
It follows that
\begin{align}
\pro{ H_{S^c\times [n]} v=0 \,\Big|\, \Sigma , (Y_i)_{i\in S}} &\le \prod_{i\in A'} \pro{Y_i\cdot (r_i\schur v ) = 0 \,\Big|\, \Sigma , (Y_i)_{i\in S}} \notag \\
& \le \left[ \frac{C}{\sqrt{pk}}\right]^{n-k - O\big(\frac{n\log n}{d}\big)}.	\label{mmk}
\end{align}
For $\frac{n\log n}{d}\ll k\le \frac{n}2$ this expression is bounded by 
$O\big(\expo{-cn\log\log n}\big)$,
which combines with \eqref{simplerpm} to give 
\begin{align}
\pro{\event_k\setminus \event_{k-1}} &\ll 4^n \expo{ -cn \log\log n}	\notag\\
&= O\left( \expo{-cn \log\log n}\right).		\label{medk}
\end{align}
For $\frac{n}2\le k\le (1-\eta)n$ we instead bound \eqref{mmk} by
\begin{equation*}
O\left(\expo{-\frac{1}2(n-k) \log(pk)} \right) = O\left(\expo{-\frac{1}2(n-k) \log d} \right) 
\end{equation*}
assuming $\eta\ge C\frac{\log n}{d}$ for $C>0$ sufficiently large. 
With \eqref{simplerpm} we conclude
\begin{align}
\pro{\event_k\setminus \event_{k-1}} &\ll {n\choose n-k}^2 \expo{ -\frac12(n-k)\log d} \notag\\
&\le \left(\frac{en}{n-k}\right)^{2(n-k)} d^{-(n-k)/2} \notag \\
&\le \left( \frac{e}{\eta d^{1/4}}\right)^{2\eta n} \notag \\
&\ll \expo{-c\eta n}		\label{bigk}
\end{align}
assuming $\eta$ is at least a sufficiently large multiple of $d^{-1/4}$. 

Summing the bounds \eqref{smallk}, \eqref{medk}, \eqref{bigk} over $1\le k\le (1-\eta)n$ completes the proof. \qed

\subsection{Preliminary reductions}	\label{sec_structprelim}
We now turn to the unsigned rrd matrix $M$ and the proof of Proposition \ref{prop_struct}.
Recall our notation for the level sets of a vector $x\in \R^n$:
$$x^{-1}(\lambda) := \set{i\in [n]: \, x(i)=\lambda}$$
for $\lambda\in \R$. 
Our aim is to show that the good event
\begin{align*}
\good^{\sls}(\eta) &:= \Big\{ \forall \lambda\in \R \mbox{ and }\forall \, 0\ne x\in \R^n: Mx=0\mbox{ or } M^\tran x=0, \, \mbox{ we have } |x^{-1}(\lambda)|\le \eta n \Big\}
\end{align*}
holds with probability $1-O(n^{-100})$ for any $\eta\in [C_1d^{-c_0},1]$, for some constants $C_1,c_0>0$.
Let
\begin{equation}
\good_R(\eta) = \Big\{\,\forall \lambda\in \R,\;  \forall\, 0\ne x\in \R^n \mbox{ such that $Mx=0$ }, \big|x^{-1}(\lambda)\big|\le \eta n \,\Big\}.
\end{equation}
Since $M\eqd M^\tran$, by a union bound it suffices to show that $\good_R(\eta)$ holds 
with probability $1-O(n^{-100})$.
The following claim recasts $\good_R(\eta)$ as the event that there is a sparse vector that is mapped by $M$ to a constant vector. 

\begin{claim}
\label{claim_goodbad}
For any $\eta\in (0,1]$, we have
\begin{equation}	\label{goodLbad}
\good_R(\eta)^c = \Big\{\, \exists\, 0\ne y\in \R^n: \, \big|\spt(y)\big|\le (1-\eta)n, \, My\in \set{0, \ones}\, \Big\}
\end{equation}
where we recall that $\ones\in \R^n$ is the vector with all components equal to $1$. 
\end{claim}
(We actually only need the containment $\subset$ in \eqref{goodLbad}.)
\begin{proof}
Let us denote the right hand side of \eqref{goodLbad} by $\bad(\eta)$.
Suppose that $\good_R(\eta)$ fails. Then there exists a nontrivial null vector $x\in \R^n$ and $\lambda\in \R$ such that $|x^{-1}(\lambda)|>\eta n$.
Let $y=\lambda\ones-x$. Then $y$ is nontrivial and $|\spt(y)|<(1-\eta)n$. Moreover,
$$My = \lambda M\ones -Mx = \lambda d\ones \in \langle\ones\rangle$$
so by dilating $y$ we see that $\bad(\eta)$ holds. 

Conversely, suppose that $\bad(\eta)$ holds. Then there exists a nontrivial vector $y$ supported on at most $(1-\eta)n$ coordinates such that $My$ is either 0 or $\ones$.
If $My=0$ then we are in $\good_R(\eta)^c$ (simply taking $x=y$ and $\lambda=0$).
So assume $My=\ones$. 
Now letting $x=y-\frac{1}{d}\ones$, we have that $x$ is a right null vector of $M$ with $|x^{-1}(1/d)|>\eta n$, so we are in $\good_R(\eta)^c$.
\end{proof}

It remains to show that $\good_R(\eta)$ holds 
with probability $1-O(n^{-100})$. 
Letting
\begin{equation}		\label{eventkdef}
\event_k:= \big\{\, \exists y\in \R^n: \, |\spt(y)|=k, \,My\in \set{0,\ones} \big\}
\end{equation}
we have
\begin{equation}	\label{decompose}
\pr\big(\good_L(\eta)^c\big) =\sum_{k=2}^{\lf(1-\eta)n\rf} \pr\big(\event_k\setminus \event_{k-1}\big)
\end{equation}
(note that $\event_1$ is empty since no column can be parallel to $0$ or $\ones$).

The following lemma is analogous to the bound \eqref{simplerpm} from the proof of Proposition \ref{prop_structpm}.
The proof is lengthier but follows similar reasoning.

\begin{lemma}[Passing to a large minor]
\label{lem_tall}
For $k\in [n]$, let
\begin{equation}
\mW_k = \Big\{ \hat{v}\in \R^k : \; v(j) \ne 0 \; \forall j\in [k] \Big\}
\end{equation}
be the set of vectors in $\R^k$ with full support. 
Suppose that for some $Q_k\ge 0$ we have a bound
\begin{equation}	\label{qk}
\pro{ M_{[k+1,n]\times [k]}\hat{v} = \alpha \ones \,\Big|\, R_1,\dots,R_k} \le Q_k
\end{equation}
that is uniform in the choice of $\hat{v}\in \mW_k$, $\alpha \in \set{0,1}$ and the realization $R_1,\dots, R_k$ of the first $k$ rows of $M$. 
Then we have
\begin{equation}
\pr\big(\event_k\setminus \event_{k-1}\big)\ll {n\choose k}^2 Q_k.
\end{equation}
\end{lemma}

\begin{proof}
By column exchangeability and a union bound, we have
\begin{equation}\label{shufflec}
\pro{\event_k\setminus \event_{k-1}}\le {n\choose k}\pro{\event_{[k]}\setminus\event_{k-1}}
\end{equation}
where $\event_{[k]}=\event_{[k]}^0\vee \event_{[k]}^1$, with 
$$
\event_{[k]}^0:=\big\{\exists x\in \R^n:\mbox{ $\spt(x)=[k]$, $Mx=0$}\big\}
$$
and
$$
\event_{[k]}^1:=\big\{\exists x\in \R^n:\mbox{ $\spt(x)=[k]$, $Mx=\ones$}\big\}.
$$
From
$
\event_{[k]}= \event_{[k]}^0\vee (\event_{[k]}^1\setminus \event_{[k]}^0)
$
we may bound
\begin{equation}\label{split}
\pro{\event_{[k]}\setminus \event_{k-1}} \le \pro{\event_{[k]}^0\setminus \event_{k-1}} + \pro{\event_{[k]}^1\setminus (\event_{[k]}^0\vee \event_{k-1})}.
\end{equation}

For the first term on the right hand side, note that on $\event_{[k]}^0\setminus \event_{k-1}$ the minor $M_{[n]\times [k]}$ has $k-1$ linearly independent rows. Indeed, if this were not the case we would have $\rank(M_{[n]\times [k]})\le k-2$, so that $M_{[n]\times [k]}$ has 2 linearly independent right null vectors $x_1,x_2\in \R^k$. But there is a $k-1$-sparse linear combination of $x_1,x_2$, putting us in $\event_{k-1}$.

For the second term in (\ref{split}), note that on the complement of $\event_{[k]}^0\vee \event_{k-1}$ the minor $M_{[n]\times[k]}$ has full rank, and hence has $k$ linearly independent rows.

Now we spend some symmetry to fix the linearly independent rows. Let $\mL_{i}$ denote the event that $R_1,\dots,R_{i}$ are linearly independent. By row exchangeability we have
\begin{equation}\label{entropy01}
\pro{\event_{[k]}^0\setminus \event_{k-1}} \le {n\choose k-1} \pro{ \left(\event_{[k]}^0\setminus \event_{k-1} \right)\wedge \mL_{k-1}}
\end{equation}
and
\begin{equation}\label{entropy11}
\pro{\event_{[k]}^1\setminus (\event_{[k]}^0\vee \event_{k-1})} \le {n\choose k} \pro{ \left(\event_{[k]}^1\setminus (\event_{[k]}^0\vee \event_{k-1}) \right)\wedge \mL_{k}}.
\end{equation}

In (\ref{entropy01}), $(\event_{[k]}^0\setminus \event_{k-1})\wedge \mL_{k-1}$ is the event that the first $k-1$ rows of $M$ are linearly independent, that there is a null vector $x$ of $M$ supported on $[k]$, and that there are no $k-1$-sparse null vectors of $M$. Now on this event there is actually only one possibility for $x$ up to dilation. Indeed, on $\mL_{k-1}$ the system
\begin{equation}\label{system1}
M_{[k-1]\times [k]}z = 0
\end{equation}
has a unique solution up to dilation, by the linear independence of the first $k-1$ rows. Let us pick a nontrivial solution $\hx\in \R^k$ of (\ref{system1}) arbitrarily, and set $x^*=(\hx\quad 0)^\tran \in \R^n$. On the complement of $\event_{k-1}$, each component of $\hx$ is nonzero. Hence, $(\event_{[k]}^0\setminus \event_{k-1})\wedge \mL_{k-1}$ is contained in the event
\begin{align*}
\event_k' &:= \mL_{k-1}\wedge \big\{\hx(j)\ne0 \mbox{ for all $j\in [k]$}\big\} \wedge \big\{Mx^*=0\big\} \\
&=\mL_{k-1}'\wedge \set{M_{[k,n]\times[k]}\hx=0 } 
\end{align*}
where we have let $ \mL_{k-1}'=\mL_{k-1} \wedge \set{ \hx(j)\ne0 \mbox{ for all $j\in [k]$} }$.
We emphasize that $\hx$ is a random vector in $\R^k$, defined only on the event $\mL_{k-1}$, and fixed by conditioning on the first $k-1$ rows of $M$ through (\ref{system1}). 

We may similarly fix the vector in the preimage of $\ones$ on the event $(\event_{[k]}^1\setminus (\event_{[k]}^0\vee \event_{k-1}))\wedge \mL_{k}$ from (\ref{entropy11}). 
This event is disjoint from the event $(\event_{[k]}^0\setminus \event_{k-1})\wedge \mL_{k-1}$ from (\ref{entropy01}), and on it we may define $\hy\in \R^k$ as the unique solution of
\begin{equation}\label{system2}
M_{[k]\times [k]}y=\ones.
\end{equation}
Setting
\begin{equation}
\event_k'' := \mL_k''\wedge \set{ M_{[k+1,n]\times [k]}\hy=\ones } 
\end{equation}
where
$$\mL_k'':=\mL_k \wedge \big\{\hy(j)\ne0 \mbox{ for all $j\in [k]$}\big\}$$
we similarly conclude that 
$$
(\event_{[k]}^1\setminus (\event_{[k]}^0\vee \event_{k-1}))\wedge \mL_{k}\subset \event_k''.
$$
Here also, $\hy\in \R^k$ is a random vector defined only on the event $\mL_k$ via (\ref{system2}), fixed by conditioning on the first $k$ rows of $M$.

Combined with (\ref{entropy01}), (\ref{entropy11}), (\ref{split}) and (\ref{shufflec}), we have
\begin{equation}\label{sparse_bound1}
\pr(\event_k\setminus \event_{k-1}) \le {n\choose k}{n\choose k-1} \pr(\event_k') +{n\choose k}{n\choose k}\pr(\event_k'').
\end{equation}
By conditioning on a realization of $R_1,\dots, R_k$ such that $\mL_k'$ holds, which fixes $\hat{x}$, we see that $\pr(\event_k')\le Q_k$, with $Q_k$ as in \eqref{qk}.
We similarly have that $\pr(\event_k'')\le Q_k$, and the result follows.
\end{proof}

The bound $Q_k$ will play the same role as bounds on $\pr(\event_{S,T})$ did in the proof 
of Proposition \ref{prop_structpm} in Section \ref{sec_structpm}.
As in that proof, our approach will be different depending on the size of $k$.
We want to control the event
\begin{equation}	\label{dcrows}
\big\{ \, M_{[k+1,n]\times[k]}\hat{v}=\alpha\ones \, \} = \bigwedge_{i=k+1}^n\big\{\, R_i\cdot v=\alpha \,\big\}
\end{equation}
where $v=(\hat{v}\quad 0)$ extends $\hat{v}$ to a vector in $\R^n$.
In Section \ref{sec_structpm} we did this by conditioning on $M$ and using the randomness of the signs.
We then viewed \eqref{dcrows} as the event that several independent random walks all landed at the same point, and used the expansion properties enjoyed by $\Sigma$ on the good events $\good_1(d),\good_2(d)$ to argue that a large number of the walks took a large number of steps. 

Here we will ``inject" random walks into the distribution of the dot products $R_i\cdot v$ by applying the shuffling couplings of Lemmas \ref{lem_shuffling} and \ref{lem_rshuffling}.
For small $k$, we will apply shufflings to pairs of \emph{columns}, which will be chosen so that the number of rows altered by the switchings is large.
Conditioning on $M$, in the randomness of the switchings we will have that the events on the right hand side of \eqref{dcrows} are independent, and have probability at most $1/2$ for the affected rows.
For large $k$ we will apply shufflings independently to several non-overlapping pairs of rows, and use Lemma \ref{lem_rwalk} to bound the probabilities of the events in \eqref{dcrows}.

By Corollaries \ref{cor_discrep} and \ref{cor_thin} we may restrict to $\good^{\ee}(\eps )$ and $\bad(\eps_0,\gamma)^c$ for some $\eps,\eps_0,\gamma>0$ to be chosen sufficiently small and independent of $n$ -- these events will play similar roles to the events $\good_2(d),\good_1(d)$, respectively, in the previous section.
By Theorem \ref{thm_codeg} we may also restrict to $\good^{\ex}(\delta)$ for some $\delta>0$ to be chosen small independent of $n$.
For now let $\eta\in (0,1]$ possibly depending on $d$.
We will put restrictions on the range of $\eta$ as the proof develops, ultimately taking 
$\eta\ge C_1d^{-c_0}$ for some constants $C_1,c_0>0$. 

\subsection{High sparsity}
\label{sec_smallk}

Fix $k\le \frac{1}{2\gamma} \frac{n\log n}{d}$. 
Towards an application of Lemma \ref{lem_tall}, we fix $\hat{v}\in \mW_k$ and $\alpha\in \set{0,1}$. 
Pair off the first $k$ columns of $M$ with the last $k$ columns according to some bijection
$$\sigma: [k]\rightarrow [n-k+1,n]$$
chosen in some arbitrary fashion, say uniformly at random and independently of $M$. 

The following lemma shows that we can locate a large number of pairs of columns $(j,\sigma(j))$ on which we can perform independent restricted shufflings (see Lemma \ref{lem_rshuffling}). 
We use restricted shufflings rather than Lemma \ref{lem_shuffling} in order to ``spread out" the switching modifications to $M$.
Specifically, we want to ensure that each row of $M$ is only affected by at most one random sign, in order to decouple the events in \eqref{dcrows}.

\begin{lemma}[Locating disjoint patches of row indices for shufflings]
\label{lem_patches}
Let $\eps_0\in (0,1)$, and assume $\eps_0,\gamma,\delta$ are sufficiently small.
Then on the event $\good^{\ex}(\delta)\wedge\bad(\eps_0,\gamma)^c$ (the former event was defined in Theorem \ref{thm_codeg} and the latter in Corollary \ref{cor_thin}),
for some $m\gg k/\log n$
there exists an increasing sequence of column indices 
\begin{equation}	\label{jseq}
1\le j_1<j_2<\cdots<j_m\le k
\end{equation}
and an increasing sequence of sets of row indices
\begin{equation}	\label{fixseq}
[k]=\Fix(1)\subset \Fix(2)\subset \cdots \subset \Fix(m)\subset  [n]
\end{equation}
such that the following properties hold:
\begin{enumerate}
\item (Patches are large) For each $\ell\in [m]$, letting
\begin{align*}
A_\ell^+ &:= \Ex_{M^\tran}(j_\ell, \sigma(j_\ell))\setminus \Fix(\ell-1)\\
A_\ell^- &:= \Ex_{M^\tran}(\sigma(j_\ell),j_\ell)\setminus \Fix(\ell-1)
\end{align*}
we have 
\begin{equation}
|A_\ell^+|,\, |A_\ell^-| \ge .01d.
\end{equation}

\item (Disjointness) The $m$ sets $\set{A_\ell^+\cup A_\ell^-}$ are pairwise disjoint.

\item (Conditioning) For each $\ell\in [m]$, $\Fix(\ell)$ is fixed by conditioning on the $2(\ell-1)$ columns $\bigcup_{\ell'<\ell} \set{j_{\ell'}, \sigma(j_{\ell'})}$.
\end{enumerate}

\end{lemma}

We defer the proof of this lemma for now and use it to bound bound $\pr(\event_k\setminus \event_{k-1})$. 
As we have already restricted to $\good^{\ex}(\delta)\wedge\bad(\eps_0,\gamma)^c$, let $m\gg k/\log n$ and the sequences $(j_\ell)_{\ell=1}^m$, $(A_\ell^\pm)_{\ell=1}^m$ and $(\Fix(\ell))_{\ell=1}^m$ as in the lemma. 
We can form a coupling $(M,\tM)$ of rrd matrices using Lemma \ref{lem_rshuffling} by performing independent restricted shufflings on $M$ at the columns $(j_\ell, \sigma(j_\ell))$.
Specifically, letting $s:= \lceil .01 d\rceil$, 
for each $\ell\in [m]$ we draw $S_\ell^+\subset A_\ell^+$, $S_\ell^-\subset A_\ell^-$ of size $s$ independently and uniformly at random, and conditional on these $2m$ sets we draw $m$ independent uniform random bijections
$\pi_\ell: S_\ell^+\rightarrow S_\ell^-.$
We let $\xi:[n]\rightarrow\set{\pm1}$ be a sequence of iid uniform signs independent of all other random variables.
Then for each $\ell\in [m]$ and each $i\in S_\ell^+$, we replace the minor $M_{(i,\pi_\ell(i))\times (j_\ell, \sigma(j_\ell))}$ with the random $2\times 2$ matrix
$$\eye \un(\xi(i)=+1)+\jay \un(\xi(i)=-1).$$ 
By the independence of the signs $\xi(i)$ and the fact that the $2m$ sets $\set{ S_\ell^+, S_\ell^-}_{\ell\in [m]}$ are pairwise disjoint, we have
\begin{align*}
\pro{ \tM_{[k+1,n]\times [k]}\hat{v} = \alpha\ones \,\Big|\, M} & \le \prod_{\ell=1}^m \prod_{i\in S_\ell^+} \pro{ \sum_{j\in [k]}\tM(i,j) v(j) = \alpha \;\bigg|\; M} \\
&\le \left(\frac12\right)^{.01dm} \\
&\le \expo{ -c dk/\log n}.
\end{align*}
Since $M\eqd\tM$, by Lemma \ref{lem_tall}, we conclude
\begin{align*}
\pro{ \event_k\setminus \event_{k-1}} &\ll {n\choose k}^2 \expo{ -c\frac{dk}{\log n}} \\
&\le \expo{ 2k\left( \log n - c\frac{d}{\log n}\right) } \\
&\le \expo{ - c \frac{dk}{\log n}}\\
&\le n^{-100k}
\end{align*}
by our assumption 
$d\ge C_0\log^2n$, taking $C_0$ sufficiently large. 
Summing the bounds \eqref{smallk} gives
\begin{equation}
\pro{\event_k} \ll n^{-100}
\end{equation}
for any $k\le \frac{1}{2\gamma}\frac{n\log n}{d}$.\\

\begin{proof}{\emph{of Lemma \ref{lem_patches}}}.
Set $A_0=[k]$.
We build the sequences $(j_\ell)_{\ell=1}^m$ and $(\Fix(\ell))_{\ell=1}^m$ by a simple greedy procedure. 
For each $\ell\ge 1$, we inductively define $j_{\ell}$ 
to be the smallest $j\in [k]$ such that both of the sets 
$$\Ex_{M^\tran}(j, \sigma(j))\setminus \Fix(\ell-1) ,\quad \Ex_{M^\tran}( \sigma(j), j)\setminus \Fix(\ell-1) $$
are of size at least $0.1d$.
Then with $A_\ell^+, A_\ell^-$ as in the statement of the lemma, we set 
$$\Fix(\ell) = \Fix(\ell-1)\cup A_{\ell}^+\cup A_{\ell}^-.$$
If no such $j_{\ell}$ exists, we set $m:=\ell-1$ and STOP. 

The resulting sequences $(j_\ell)_{\ell=1}^m$ and $(\Fix(\ell))_{\ell=1}^m$ clearly satisfy the three properties in the statement of the lemma. 
It only remains to show that the halting time $m$ of the greedy procedure is of size $\Omega(k/\log n)$ if we take $\eps_0,\gamma$ sufficiently small.

We abbreviate 
$$\Ex^+(j):= \Ex_{M^\tran}(j, \sigma(j)), \quad \Ex^-(j):= \Ex_{M^\tran}(\sigma(j), j).$$
We have that for all $j\in [k]$, either $|\Ex^+(j)\setminus \Fix(m)|$ or $|\Ex^-(j) \setminus \Fix(m)|$ is  $< .01d$. 
For each $j\in [k]$, put $j\in S$ if $|\Ex^+(j)\setminus \Fix(m)|<.01d$ and otherwise put $\sigma(j)\in S$, so that 
$|S|=k$. 

Taking $\delta $ sufficiently small, by our restriction to $\good^{\ex}(\delta )$ we may assume 
$$|\Ex^+(j)|, |\Ex^-(j)| \ge .1d$$
for all $j\in [k]$. 
It follows that at least one of $\Ex^+(j) \cap \Fix(m)$ , $\Ex^-(j) \cap \Fix(m)$ is of size at least $.09 d$.
Since $\Ex^+(j)\subset \mN_{M^\tran}(j)$, $\Ex^-(j)\subset \mN_{M^\tran}(\sigma(j))$, we have
$$\left|\mN_M(j)\cap \Fix(m) \right| \ge .09d$$
for all $j\in S$.
Now 
$$e_M(\Fix(m), S) = \sum_{j\in S} \left|\mN_M(j)\cap \Fix(m) \right|  \ge .09d|S|$$
so taking $\eps_0<.09$ and $\gamma$ sufficiently small, by our restriction to $\bad(\eps_0,\gamma)^c$ we must have
$$|\Fix(m)| \ge \frac{\eps_0 \gamma}{\log n}dk.$$
But by the inductive procedure to produce $\Fix(m)$ we have
$|\Fix(m)| \le k+ 2md$,
from which it follows that
$$m\gg \frac{k}{d}\left( \frac{d}{\log n} -1 \right) \gg \frac{k}{\log n}$$
by our assumption $d\ge C_0\log^2n$
(here we only need $d\ge 2\log n$, say).
\end{proof}

\subsection{Moderate sparsity}

Now we fix $k$ in the range $\big[ \frac{1}{2\gamma}\frac{n\log n}{d},  (1-\eta)n\big]$. 

The proof mirrors the proof for large $k$ for the Hadamard product $H=\Sigma\schur \Xi$ in Section \ref{sec_structpm}.
The general idea is to express the event that $M_{[k+1,n]\times [k]}\hat{v}=\alpha \ones$ as the event that several independent random walks all land at $\alpha$.
Without the iid signs enjoyed by $H$ we must use the shuffling coupling of Lemma \ref{lem_shuffling} to create random walks. 
We use the discrepancy property enforced by our restriction to the event $\good^{\ee}(\eps )$ (from Corollary \ref{cor_discrep}) to argue that these walks take many steps (in particular we will need an extension of Lemma \ref{lem_Aeps} used in Section \ref{sec_structpm}),
at which point we can apply the anti-concentration bound from Theorem \ref{thm_erdos} to each walk.

More precisely, we will fix disjoint sets of row indices $A_1,A_2\subset A:=[k+1,n]$ of equal size $a_1=|A_1|=|A_2|\gg n-k$, and pair off the elements of $A_1$ with those of $A_2$ according to a bijection $\sigma:A_1\rightarrow A_2$.
For each $i\in A_1$, we perform a shuffling on $M$ at the row pair $(i,\sigma(i))$; we do this independently for each $i\in A_1$ and denote the new matrix by $\tM$. 
We have
\begin{align*}
\pro{M_{[k+1,n]\times [k]} \hat{v}=0 \,\Big|\, R_1,\dots, R_k} = \e_{R_{k+1},\dots, R_{n}} \pro{ \tM_{[k+1,n]\times [k]} \hat{v} = \alpha \ones \Big| \, M}
\end{align*}
so it suffices to bound 
\begin{equation}		\label{splitwalks}
\pro{ \tM_{[k+1,n]\times [k]} \hat{v}= \alpha \ones \Big| \, M} \le \prod_{i\in A_1} \pro{ \tR_i\cdot v = \alpha \big| M}.
\end{equation}

As in Section \ref{sec_unstruct}, in order to bound the probabilities in \eqref{splitwalks} using Theorem \ref{thm_erdos}, we will need to argue that many of these random walks take many steps.
For this we take the pairing $\sigma$ to be \emph{random} -- it is then possible to show using our restriction to the edge discrepancy event $\good^{\ee}(\eps )$ that with overwhelming probability most of the pairs $(i,\sigma(i))$ give walks that take a large number of steps.

We turn to the details. 
Fix disjoint sets $A_1,A_2\subset A:= [k+1,n]$ with 
$a_1:= |A_1|= |A_2| \gg n-k.$
We create a new rrd matrix $\tM$ coupled to $M$ from three additional sources of randomness:
\begin{enumerate}
\item a uniform random bijection $\sigma:A_1\rightarrow A_2$ independent of all other variables;
\item a sequence $(\pi_i)_{i\in A_1}$ of uniform random bijections 
$$\pi_i: \Ex_M(i,\sigma(i))\rightarrow \Ex_M(\sigma(i),i)$$
which are jointly independent conditional on $M$ and $\sigma$;
\item an array $\Xi:[n]^2\rightarrow \set{\pm1}$ of iid uniform random signs independent of all other variables. 
\end{enumerate}
Let $\xi_i=\Xi(i,\cdot)$ denote the $i$th row of the array of signs. 
We form $\tM$ by performing a shuffling on $M$ at $(i,\sigma(i))$ according to $\pi_i$ and $\xi_i$ for each $i\in A_1$. 
We have $\tM\eqd M$ by Lemma \ref{lem_shuffling} and independence.

Recall the notation $\Steps$ from \eqref{stepsdef}, and 
for fixed $i\in A_1$ denote
\begin{align*}
\Steps_i(\hat{v}) &:= \Steps_{M, \pi_i}^{(i,\sigma(i))}(v) \\
&= \set{ j\in \Ex_M(i,\sigma(i)): v(j) \ne v(\sigma(j))}
\end{align*}
where we recall $v=(\hat{v}\quad 0)\in \R^n$ 
with $\hat{v}\in \R^k$.
Now since $\spt(v)=[k]$, we have that for each $i\in A_1$, 
\begin{equation}	\label{stepscross}
\left|\Steps_i(\hat{v}) \right| \ge \left|\Cross_i(k) \right|
\end{equation}
where we define 
\begin{align}
 \Cross_i(k) &= \Cross_{M,\pi_i}^{(i,\sigma(i))}(k) \notag \\
 &:= \set{j\in \Ex_M(i,\sigma(i)) \cap [k]: \pi_i(j) \in [k+1,n]}	\label{crossdef}
\end{align}
the number of pairs $(j,\pi_i(j))$ which are in $[k]\times [k+1,n]$, i.e.\ pairs which cross the partition $[n]=[k]\cup [k+1,n]$ going from left to right. 
(We could also include pairs crossing right to left, but this will tend to improve the lower bound \eqref{stepscross} by only a constant factor.)

Hence, for $m\ge 1$, defining 
the good events
\begin{equation}	\label{goodi}
\good_i (m) := \set{ \big| \Cross_i(k) \big| \ge m}
\end{equation}
for each $i\in A_1$, by Lemma \ref{lem_rwalk} we have
\begin{equation}	\label{rowwalki}
\pr_{\xi_i}\Big\{ \tR_i\cdot v = \alpha \Big\}1_{\good_i(m)^c} =O(m^{-1/2}).
\end{equation}

In the remainder of the proof, we show that with overwhelming probability in the randomness of the bijections $\sigma$ and $(\pi_i)_{i\in A_1}$, for most $i\in A_1$ and for a reasonably large value of $m$, $\good_i(m)$ holds except on an exponentially small event. 
(Hence we are done with the iid signs $\Xi$.)
The randomness of $M$ will only enter through our restriction to the events $\good^{\ee}(\eps )$ and $\good^{\ex}(\delta)$. 

Lemma \ref{lem_ASeps} below summarizes what we need from the discrepancy property enforced on $\good^{\ee}(\eps )$ -- it is an extension of Lemma \ref{lem_Aeps} from the proof for $H$.
While for $H$ it was enough to know that the intersections $B(i)$ of a large set $B$ with the neighborhoods $\mN_M(i)$ were of size roughly $p|B|$, here we will need intersections of $B$ with the sets $\Ex_M(i_1,i_2)$, $\Ex_M(i_2,i_1)$ to be at least a constant factor of their expected size.

For $\eps\in (0,1)$ and a set of column indices $B\subset [n]$, say that an ordered pair $(i_1,i_2)$ of distinct row indices in $A$ is \emph{$\eps$-bad for $B$} if either 
\begin{equation}
|\Ex_M(i_1,i_2)\cap B| \le \eps p|B| \quad \mbox{ or } \quad |\Ex_M(i_2,i_1)\cap B^c| \le \eps p (n-|B|).
\end{equation}
The following lemma shows that on $\good^{\ee}(\eps)$ with $\eps$ sufficiently small, 
only a small number of pairs of elements of $A=[k+1,n]$ are $\eps$-bad for $[k]$.

\begin{lemma}			\label{lem_ASeps}
Let $B\subset [n]$, and continue to denote $A=[k+1,n]$.
For $i\in [n]$, denote $B(i):= \mN_M(i)\cap B$. 
For $\eps\in (0,1)$, define
\begin{equation}
A_{\eps} = \set{ i\in A: \; \left| \frac{|B(i)|}{p |B|} - 1\right| \le \eps, \; \left| \frac{ |B^c (i)|}{p (n-|B|)}-1\right| \le \eps},
\end{equation}
and for $i\in A$, let
\begin{equation}	\label{defseps}
S_{\eps}(i) = \set{ i'\in A_{\eps}: \, (i,i') \mbox{ is $\eps$-bad for $B$}}.
\end{equation}
On the event $\good^{\ee}(\eps)$ from Corollary \ref{cor_discrep} we have
\begin{equation}	\label{Abound}
|A\setminus A_{\eps}|\ll_\eps p^{-1}\log n
\end{equation}
and for every $i\in A_\eps$,
\begin{equation}		\label{Sbound}
|S_{\eps}(i)| \ll_\eps p^{-1}\log n
\end{equation}
assuming $|B|,|B^c|\ge \frac{1}{2\gamma}\pe^{-1}\log n$ for $\gamma$ sufficiently small depending on $\eps$.
\end{lemma}

\begin{proof}
We begin with \eqref{Abound}.

Define the sets
\begin{align*}
S_1 &= \set{i\in A: |B(i)| < (1-\eps) p|B|} \\
S_2 &= \set{i\in A: |B(i)| > (1+\eps) p|B|} \\
S_3 &= \set{i\in A: |B^c(i)| < (1-\eps) p(n-|B|)} \\
S_4 &= \set{i\in A: |B^c(i)| > (1+\eps) p(n-|B|)} 
\end{align*}
so that $A\setminus A_\eps = \bigcup_{k=1}^4 S_k.$
By the same lines as the proof of Lemma \ref{lem_Aeps} we have
$|S_1| \ll_\eps p^{-1}\log n.$
By replacing $B$ with $B^c$ we obtain the same bound on $|S_3|$. 
$|S_2|$ and $|S_4|$ are bounded similarly. 

We turn to the estimate \eqref{Sbound}.
Fix $i\in A_\eps$. 
We can write $S_\eps(i) = S_\eps^1(i)\cup S_\eps^2(i)$
where 
\begin{align*}
S_\eps^1(i) &= \big\{ i'\in A_\eps:\, |\Ex(i,i')\cap B| \le \eps p|B| \big\} \\
S_\eps^2(i) &= \big\{ i'\in A_\eps:\, |\Ex(i',i)\cap B^c| \le \eps p(n-|B|) \big\}.
\end{align*}
We first bound $|S_\eps^1(i)|$. 
For $i'\in S_\eps^1(i)$, we have
\begin{equation}	\label{bbi}
|\Ex(i,i')\cap B|\le \eps p|B| \le \frac{\eps}{1-\eps}|B(i)|
\end{equation}
since $i\in A_\eps$.
It follows that
\begin{align}
e_M\big(S_\eps^1(i), B(i)\big) &= \sum_{i'\in S_\eps^1(i)}\big|\Co(i,i')\cap B\big| \notag\\
&= \sum_{i'\in S_\eps^1(i)}|B(i)| - |\Ex(i, i')\cap B| \notag\\
&\ge |S_\eps^1(i)| \left(1-\frac{\eps}{1-\eps}\right)|B(i)| \notag\\
&= \frac{1-2\eps}{1-\eps}|S_\eps^1(i)||B(i)|.	\label{undere}
\end{align}
Now we show this contradicts our restriction to the event $\good^{\ee}(\eps)$ if $\eps$ is sufficiently small. 
Recall the family $\mF(\eps)$ of pairs of subsets of $[n]$ defined in Corollary \ref{cor_discrep}.
If $\big(S_\eps^1(i), B(i)\big)\in \mF(\eps)$ we have
\begin{equation}	\label{fromthis}
e_M(S_\eps^1(i), B(i)) \le (1+\eps) p |S_\eps^1(i)||B(i)|.
\end{equation}
From \eqref{fromthis} it follows that 
$$e_M(S_\eps^1(i), B(i)) \le \frac{(1+\eps)}{2} |S_\eps^1(i)||B(i)|$$
which contradicts \eqref{undere} if $\eps$ is a sufficiently small absolute constant. 
We may hence assume $(S_\eps^1(i), B(i))\notin \mF(\eps)$.
Similarly to how we argued in the bound for $|S_1|$, we can deduce from the lower bound 
$$|B(i)|\gg p|B| \gg \gamma^{-1}\log n$$
(since $i\in A_\eps$) that taking $\gamma$ sufficiently small, we must have $|S_\eps^1(i)|\le |B(i)|$ (for $n$ sufficiently large), and hence 
$$|S_\eps^1(i)|  \ll_\eps p^{-1}\log n.$$
The proof that $|S_\eps^2(i)|=O_\eps(p^{-1}\log n)$ follows similar lines and is omitted. 
\end{proof}

We define the subset of $A_1$ of ``good" row indices to be
\begin{equation}
A_1' = \Big\{ i\in A_1\cap A_\eps: \sigma(i) \in A_\eps \setminus S_\eps(i)\Big\}
\end{equation}
where $S_\eps(i)$ is as in \eqref{defseps} with $B=[k]$.
That is, $A_1'$ is the set of $i\in A_1$ such that $i$ and $\sigma(i)$ are both in $A_\eps$, and such that the pair $(i,\sigma(i))$ is not bad for $[k]$. 
Note that this is a random set depending on $M$ and $\sigma$. 
We can now use Lemma \ref{lem_ASeps} and the randomness of $\sigma$ to show that with overwhelming probability, $A_1'$ constitutes most of $A_1$. 

Let 
\begin{equation}
\bad' = \Big\{\, |A_1\setminus A_1'| \ge |A_1|/2 \, \Big\}.
\end{equation}
Now for arbitrary $m\ge 1$ we have
\begin{align}
\pro{ \tM_{[k+1,n]\times [k]} \hat{v}= \alpha \ones \,\Big|\, M} &\le \pros[\sigma]{ \bad'} +  \e_\sigma 1_{\bad'^c} \pro{ \tM_{[k+1,n]\times [k]}\hat{v}= \alpha \ones \,\Big|\, M, \sigma} 	\notag\\
&\le \pros[\sigma]{ \bad'} + \e_\sigma 1_{\bad'^c} \prod_{i\in A_1'} \pro{ \tR_i\cdot v=\alpha \big| M, \sigma}		\notag\\
&\le \pros[\sigma]{ \bad'} + \e_\sigma 1_{\bad'^c} \prod_{i\in A_1'} \left[ \pros[\pi_i]{\good_i(m)^c} + \pr_{\xi_i}\big\{\tR_i\cdot v=\alpha\big\} 1_{\good_i(m)} \right].		\label{summarize}
\end{align}
The term $\pr_{\xi_i}\big\{\tR_i\cdot v=\alpha\big\} 1_{\good_i(m)}$ is $O(m^{-1/2})$ by \eqref{rowwalki}. 
It remains to bound $ \pros[\sigma]{ \bad'} $
and $\pros[\pi_i]{\good_i(m)^c}$ (for some large value of $m$).

From Lemma \ref{lem_ASeps} 
with $B=[k]$ 
we have
\begin{equation}
|A\setminus A_\eps|\,,\; \max_{i\in A_1\cap A_\eps}|S_\eps(i)| \le s_0
\end{equation}
for some $s_0=O(p^{-1}\log n)$ (assuming $\eta\ge \frac{1}{2\gamma}\frac{\log n}{d}$).
By crudely estimating the number of bad realizations of $\sigma$, we can bound
$$
\pr_\sigma\big(\bad'\big) \le {a_1\choose \lf a_1/2\rf}\frac{s_0^{ a_1/2}\lf a_1/2\rf!}{a_1!} 
$$
(first fixing the $\lf a_1/2\rf$ elements of $A_1\setminus A_1'$, then choosing from the at most $s_0$ options for $\sigma(i)$ for each $i\in A_1\setminus A_1'$).
Simplifying this expression and applying the inequality $n!\ge  (n/e)^n$, true for all $n\in \N$,
\begin{equation} \label{badsigbd}
\pr_\sigma\big(\bad'\big) \le  \frac{s_0^{a_1/2}}{\lf a_1/2\rf!} \le \left(\frac{Cs_0}{a_1}\right)^{a_1/2}.
\end{equation}

Now we estimate the terms $\pros[\pi_i]{\good_i(m)^c}$ (see \eqref{goodi} for the definition of these events).
For fixed $i\in A_1'$ we have
\begin{equation}	\label{ecross}
\e_{\pi_i} |\Cross_i(k)| = \frac{ \big|\Ex_M(i,\sigma(i))\cap [k]\big| \big|\Ex_M(\sigma(i),i)\cap[k+1,n]\big|}{\big|\Ex_M(i,\sigma(i))\big|}.
\end{equation}
From our restriction to $\good^{\ex}(\delta )$ we know the denominator is of size $\Theta_{\delta }(d)$, 
and since $i\in A_1'$ the numerator is of size $\Omega\left( p^2 k(n-k)\right)$
whence,
\begin{equation}
\e_{\pi_i} |\Cross_i(k)| \gg \frac{d}{n}\frac{k(n-k)}{n} \gg p \min(k,n-k).
\end{equation}
From Lemma \ref{lem_com} it follows that 
\begin{equation}
|\Cross_i(k)|\gg p \min(k,n-k)
\end{equation}
except with probability at most $\expo{ -c p\min(k,n-k)}$ in the randomness of $\pi_i$. 
We have hence shown that for $i\in A_1'$, 
\begin{equation}		\label{badpibd}
\pros[\pi_i]{\good_i(m )^c} \le \exp\big(-cp\min(k,n-k)\big)
\end{equation}
where set $m:=c p \min(k,n-k)$,
and $c>0$ is a sufficiently small absolute constant. 
In particular, this bound is of lower order than the bound $\pr_{\xi_i}\big\{\tR_i\cdot v=\alpha\big\} 1_{\good_i(m)} = O(m^{-1/2})$. 

Substituting our bounds \eqref{rowwalki}, \eqref{badsigbd} and \eqref{badpibd} into \eqref{summarize}, we have
\begin{align*}
\pro{\tM_{[k+1,n]\times [k]}\hat{v} = \alpha \ones \,\Big|\, M} &\le \pr_\sigma\big(\bad'\big) + \e_\sigma 1_{\bad'^c} O(m^{-1/2})^{|A_1'|}\\
&\le  \left( \frac{Cn\log n}{ d(n-k)}\right)^{\kappa a_1} + m^{-(1-\kappa+o(1))a_1/2}.
\end{align*}
Applying Lemma \ref{lem_tall} we have
\begin{equation}	\label{splitk}
\pro{\event_k\setminus \event_{k-1}} \ll \mbox{(I)}_k + \mbox{(II)}_k
\end{equation}
where
\begin{align*}
\mbox{(I)}_k &=  {n\choose k}^2  \left( \frac{Cn\log n}{d(n-k)}\right)^{\kappa a_1} 	\\
\mbox{(II)}_k &= {n\choose k}^2 m^{-(1-\kappa+o(1))a_1/2}.
\end{align*}
First assume $\frac{n\log n}{d}\ll k\le \frac{n}{2}.$
In this case we have
\begin{equation}	\label{mlb1}
m=cp k \gg \log n.
\end{equation}
and
\begin{equation}	\label{a1lb}
a_1\gg n-k \ge n/2
\end{equation}
so
\begin{align*}
\mbox{(I)}_k &\le 4^n \left( \frac{Cn\log n}{ d(n-k)}\right)^{ a_1/2}  \\
&\le 4^n \left( \frac{C\log n}{ d}\right)^{cn} \\
&\le e^{-c n}
\end{align*}
if $d\ge C'\log n$ for some $C'$ sufficiently large. 
For the second term,
\begin{align*}
\mbox{(II)}_k &\le 4^n m^{-(1/2+o(1))a_1/2} = (\log n)^{-\Omega(n)}
\end{align*}
by the lower bounds \eqref{mlb1}, \eqref{a1lb}. 
From these bounds and \eqref{splitk} we conclude
\begin{equation}	\label{kmid}
\pro{\event_k\setminus \event_{k-1}} \le e^{-cn}
\end{equation}
for $\frac{n\log n}{d}\ll k\le n/2.$

Now assume $n/2 \le k \le (1-\eta)n$. 
In this case we have
\begin{equation}
m=cp(n-k) = c d\frac{n-k}{n} \gg \eta d.
\end{equation}
Since $a_1=|A_1|=|A_2|$, $A_1,A_2\subset A$ are arbitrary disjoint subsets, and $|A|=n-k$, we may take $a_1=(\frac12 -o(1))(n-k)$.
We then have
\begin{align*}
\mbox{(I)}_k &\le \left( \frac{en}{n-k}\right)^{2(n-k)} \left( \frac{Cn\log n}{ d(n-k)}\right)^{a_1/2}  \\
&\le \left[ \frac{C\log n}{d}\left( \frac{n}{n-k}\right)^{9}\right]^{(\frac14-o(1))(n-k)} \\
&\le \left( \frac{C\log n}{d\eta^{9}}\right)^{c(n-k)}.
\end{align*}
By our assumption $d\ge C_0\log^2n$ we conclude
\begin{equation}	\label{final1}
\mbox{(I)}_k \le \left( \frac{C}{d^{1/2}\eta^9}\right)^{c(n-k)} \le \left( \frac{C}{d\eta^{18}}\right)^{c'(n-k)}.
\end{equation}
For the other term:
\begin{align}
\mbox{(II)}_k &\le \left( \frac{en}{n-k}\right)^{2(n-k)}  m^{-(1/2+o(1))a_1/2} \notag \\
&\le \left( \frac{ C}{d\eta^{17+o(1)}}\right)^{c(n-k)}. \label{final2}
\end{align}
Combining the bounds \eqref{final1} and \eqref{final2}, we have that for $k\in [\frac{n}2,(1-\eta)n]$, 
\begin{equation}	\label{khigh}
\pro{\event_k\setminus \event_{k-1}} \le e^{-c\eta n}
\end{equation}
if we assume 
$\eta\ge  C_1d^{-1/18}$ for a sufficiently large constant $C_1>0$. 

Summing the bounds \eqref{klow}, \eqref{kmid}, \eqref{khigh} over their respective ranges of $k$, we conclude 
$$\pro{ \event_{\lf (1-\eta)n\rf}} \ll n^{-100}$$
as desired. \\

\noindent{\bf Acknowledgements}
The author thanks Terence Tao for invaluable discussions on this problem and on random matrix theory in general, as well as for helpful feedback on preliminary versions of the manuscript.
Thanks also go to 
Ioana Dumitriu and Jamal Najim for the suggestion to consider signed rrd matrices, in large part because the proof of Theorem \ref{thm_pm} inspired arguments to improve the main theorem, allowing the degree $d$ to lower from $n^{1/2+\eps}$ to $C\log^2n$.
Finally, the author is grateful to the anonymous referees for their careful reading and numerous corrections and suggestions to improve the manuscript.

\bibliographystyle{abbrv}
\bibliography{biblio_sing.bib}   

\end{document}